\documentclass{amsart}

\usepackage{color,graphicx,amssymb,latexsym,amsfonts,txfonts,amsmath,amsthm}
\usepackage{pdfsync}
\usepackage[all,cmtip]{xy}
\usepackage{subfigure}
\usepackage{tikz}
\usetikzlibrary{matrix}

\usepackage{hyperref}
\hypersetup{
    colorlinks=true,       % false: boxed links; true: colored links
    linkcolor=blue,          % color of internal links
    citecolor=blue,        % color of links to bibliography
    filecolor=blue,      % color of file links
    urlcolor=blue           % color of external links
}

\newcommand{\be}{\begin{equation}}
\newcommand{\ee}{\end{equation}}
\numberwithin{equation}{section}

\begin{document}

\definecolor{Rojo}{rgb}{1,0,0}

\newtheorem{theo}{Theorem}              
\newtheorem{propo}{Proposition}                  
\newtheorem{coro}{Corollary}
\newtheorem{lemm}{Lemma}

\theoremstyle{remark}         
\newtheorem{rema}{\bf Remark}
\newtheorem{example}{\bf Example}
\newtheorem{nota}{\bf Notation}

\title{On the connectivity of the branch and real locus of ${\mathcal M}_{0,[n+1]}$}

\author{Yasmina Atarihuana and  Rub\'en A. Hidalgo}

\address{Departamento de Matem\'atica y Estad\'{\i}stica, Universidad de la Frontera. Temuco, Chile}
\email{yasminafernanda@hotmail.com, ruben.hidalgo@ufrontera.cl}

\thanks{Partially supported by projects Fondecyt 1190001 and Anillo ACT1415 PIA-CONICYT}

%%%%%%%%%%%%%%%%%%%%%%
%%%%%%%%%%%%%%%%%%%%%%
\begin{abstract}
If $n \geq 3$, then moduli space ${\mathcal M}_{0,[n+1]}$, of isomorphisms classes of $(n+1)$-marked spheres, is a complex orbifold of dimension $n-2$. 
Its branch locus ${\mathcal B}_{0,[n+1]}$ consists of the isomorphism classes of those $(n+1)$-marked spheres with non-trivial group of conformal automorphisms. We prove that
${\mathcal B}_{0,[n+1]}$ is connected if either $n \geq 4$ is even or if $n \geq 6$ is divisible by $3$, and that it has exactly two connected components otherwise.  The orbifold ${\mathcal M}_{0,[n+1]}$ also admits a natural real structure, this being induced by the complex conjugation on the Riemann sphere. The locus ${\mathcal M}_{0,[n+1]}({\mathbb R})$ of its fixed points, the real points, consists of the isomorphism classes of those marked spheres admitting an anticonformal automorphism. Inside this locus is the real locus ${\mathcal M}_{0,[n+1]}^{\mathbb R}$, consisting of those classes of marked spheres admitting an anticonformal involution. We prove that ${\mathcal M}_{0,[n+1]}^{\mathbb R}$ is connected for $n \geq 5$ odd, and that it is disconnected for $n=2r$ with $r \geq 5$ is odd. 
\end{abstract}

\maketitle

%%%%%%%%%%%%%%%%%%
%%%%%%%%%%%%%%%%%%
\section{Introduction}
Let $g \geq 0$ and $n\geq -1$ be integers such that $3g-2+n>0$ (so, for $g=0$, we have $n \geq 3$). The moduli space ${\mathcal M}_{g,[n+1]}$, of isomorphism classes of $(n+1)$-marked Riemann surfaces of genus $g$, is a complex orbifold of dimension $3g-2+n$. Its branch locus ${\mathcal B}_{g,[n+1]} \subset {\mathcal M}_{g,[n+1]}$ consists of the isomorphism classes of those admitting non-trivial conformal automorphisms. In \cite{BCI} it was proved that ${\mathcal B}_{g,[0]} \subset {\mathcal M}_{g,[0]}={\mathcal M}_{g}$ is connected only for $g \in \{3,4,13,17,19,59\}$. The complex orbifold ${\mathcal M}_{g,[n+1]}$ also admits a natural antiholomorphic automorphism of order two (a real structure) which is induced by the usual complex conjugation. The locus ${\mathcal M}_{g,[n+1]}({\mathbb R}) \subset {\mathcal M}_{g,[n+1]}$ of fixed points (the real points) of such a real structure consists of the isomorphic classes of those admitting anticonformal automorphisms. Let  ${\mathcal M}_{g,[n+1]}^{\mathbb R} \subset {\mathcal M}_{g,[n+1]}({\mathbb R})$ be the sublocus of those classes having a representative admitting an anticonformal involution (equivalently, the representative being definable over the reals). In \cite{BSS,CI,Seppala} it has been proved that ${\mathcal M }_{g,[0]}^{\mathbb R} \subset {\mathcal M}_{g}$ is connected. In \cite{CH} it was proved that ${\mathcal M}_{g,[0]}({\mathbb R}) \subset {\mathcal M}_{g}$ is also connected but that ${\mathcal M}_{g,[0]}({\mathbb R})\setminus {\mathcal M }_{g,[0]}^{\mathbb R}$ is not in general connected. 
In this paper, for each $n \geq 3$, we study the connectivity of both ${\mathcal B}_{0,[n+1]}$ and ${\mathcal M }_{0,[n+1]}^{\mathbb R}$.  

Torelli space ${\mathcal M}_{0,n+1}$ is the moduli space of isomorphisms classes of ordered $(n+1)$-marked spheres. The mapping class group ${\rm Mod}_{0,[n+1]}$ induces an action of the symmetric group ${\mathfrak S}_{n+1}$ as a group ${\mathbb G}_{n}$ of holomorphic automorphisms of ${\mathcal M}_{0,n+1}$ (called the Torelli group) and ${\mathcal M}_{0,[n+1]}={\mathcal M}_{0,n+1}/{\mathfrak S}_{n+1}$.
If $n=3$, then ${\mathcal M}_{0,4}$ can be identified with the orbifold whose underlying space is $\Omega_{3}={\mathbb C}\setminus\{0,1\}$ and all of its points being conical points of order $4$. In this case, 
${\mathbb G}_{3} \cong {\mathfrak S}_{3}$ (the action of ${\mathfrak S}_{4}$ on ${\mathcal M}_{0,4}$ is not faithful as it contains a normal subgroup $K_{3} \cong C_{2}^{2}$ acting trivially). In particular,  ${\mathcal B}_{0,[4]}={\mathcal M}_{0,[4]}$. The quotient 
 orbifold ${\Omega}_{3}/{\mathbb G}_{3}$ can be identified with the complex plane ${\mathbb C}$ with two cone points, one of order two and the other of order three, (the two cone points corresponds exactly to those $4$-marked spheres whose of conformal automorphisms is bigger than $C_{2}^{2}$). Also,  ${\mathcal M}_{0,[4]}^{\mathbb R}={\mathbb R}$. 
If $n \geq 4$, then ${\mathcal M}_{0,n+1}$ can be identified with 
the domain $\Omega_{n} \subset {\mathbb C}^{n-2}$ consisting of those tuples $(z_{1},\ldots,z_{n-2})$, where $z_{j} \in \Omega_{3}$ and $z_{i} \neq z_{j}$ for $i \neq j$. In this case, ${\mathbb G}_{n} \cong {\mathfrak S}_{n+1}$ acts faithfully as the full group of holomorphic automorphisms of ${\mathcal M}_{0,n+1}$ \cite{Royden, EK} and 
$\Omega_{n}/{\mathbb G}_{n}={\mathcal M}_{0,[n+1]}$. If ${\rm Sing}_{0,[n+1]} \subset {\mathcal M}_{0,[n+1]}$ is the locus of non-manifold points, then: (i) for $n \geq 6$, ${\rm Sing}_{0,[n+1]}={\mathcal B}_{0,[n+1]}$ \cite{Patterson} and (ii) for $n \in \{4,5\}$, the singular locus consists of exactly one point \cite{Igusa}. If, for $T \in {\mathbb G}_{n}\setminus\{I\}$, we denote by ${\rm Fix}(T) \subset \Omega_{n}$ the locus of its fixed points, then in \cite{Schneps}  it was observed that, for ${\rm Fix}(T) \neq \emptyset$ (which might not be connected), its projection to ${\mathcal M}_{0,[n+1]}$ is connected. We obtain the following connectivity of ${\mathcal B}_{0,[n+1]}$, whose proof is provided in Section \ref{Sec:pruebateo1}.

\begin{theo}\label{branchconexo}
The branch locus ${\mathcal B}_{0,[n+1]}$ is connected if either (i) $n \geq 4$ is even or (ii) $n \geq 6$ is divisible by $3$. It has exactly two connected components otherwise.
\end{theo}

In Section \ref{simetrias} we deal with the locus ${\mathcal M}_{0,[n+1]}({\mathbb R})$, which is given by the projection of those points being fixed by some antiholomorphic automorphism of $\Omega_{n}$ (the real locus ${\mathcal M}_{0,[n+1]}^{\mathbb R}$ is the projection of those points in $\Omega_{n}$ being fixed by a symmetry, that is, an antiholomorphic involution). So, the points in the complement ${\mathcal M}_{0,[n+1]}({\mathbb R}) \setminus {\mathcal M}_{0,[n+1]}^{\mathbb R}$ are the isomorphic classes of those points having antiholomorphic automorphisms but no symmetries in their stabilizers.
The next summarizes the main results at this point.

\begin{theo}\label{teoreal}
If $n \geq 4$, then the following hold.
\begin{enumerate}
\item The space $\Omega_{n}$ has exactly $[(n+3)/2]$ symmetries.
\item The locus of fixed points of a symmetry of $\Omega_{n}$ is non-empty and each of its connected components is a real submanifold of real dimension $n-2$.
\item If $n$ is even, then the projection in ${\mathcal M}_{0,[n+1]}$ of the locus of fixed points of a symmetry of $\Omega_{n}$ is a connected real orbifold of dimension $n-2$.
If $n$ is odd, then the same holds for a symmetry, with the exception of those conjugated to 
$$S(z_{1},\ldots,z_{n-2})=\left(\overline{z}_{1},\frac{\overline{z}_{1}}{\overline{z}_{3}},\frac{\overline{z}_{1}}{\overline{z}_{2}},\frac{\overline{z}_{1}}{\overline{z}_{5}},\frac{\overline{z}_{1}}{\overline{z}_{4}},\ldots,\frac{\overline{z}_{1}}{\overline{z}_{n-2}},\frac{\overline{z}_{1}}{\overline{z}_{n-3}}  \right),$$   for which the projection of its fixed points has two connected components, each one intersecting the projection of fixed points of the symmetry $J(z_{1},\ldots,z_{n-2})=(\overline{z}_{1},\ldots,\overline{z}_{n-2})$.
\item The real locus ${\mathcal M}_{0,[n+1]}^{\mathbb R}$ is connected for $n \geq 5$ odd and 
it is not connected for $n=2r$, $r \geq 5$ odd. If $p \geq 5$ is a prime, then ${\mathcal M}_{0,[2p+1]}^{\mathbb R}$ has exactly $(p-1)/2$ connected components.
\end{enumerate}
\end{theo}

Part (2) is given by Proposition \ref{novacio} and part (3) by Proposition \ref{reales}. The projection of the locus of fixed points of a symmetry is called an irreducible component of ${\mathcal M}_{0,[n+1]}^{\mathbb R}$ (some special care must be taken for the special type of symmetry $S$ as described in part (3)). Proposition \ref{interseccion} states a necessary and sufficient condition for two irreducible components to intersect, which permits to obtain part (4) (Propositions \ref{nimpar} and  \ref{nparporimpar}). In the case that $n=4p$, where $p \geq 2$ is a prime, the conditions of Proposition \ref{interseccion} permits to check that
${\mathcal M}_{0,[n+1]}^{\mathbb R}$ is connected when $p \in \{2,3,5\}$ but it is disconnected for $p \geq 7$ (Proposition \ref{caso4p}). For $n=4r$, where $r \geq 1$ is odd but different from a prime, it happens that ${\mathcal M}_{0,[n+1]}^{\mathbb R}$ is connected for $r\in \{1,9,15,21,27,33\}$ and it is not connected for $r\in\{25,35\}$.

%%%%%%%%%%%%%%%%%%%%%%%%%%%%%%%%
\subsection{Some applications}
\subsubsection{Hyperelliptic Riemann surfaces}
If $n=2g+1$, where $g \geq 2$, then the moduli space ${\mathcal M}_{0,[n+1]}$ can be identified with the moduli space ${\mathcal H}_{g}$ of hyperelliptic Riemann surfaces of genus $g$. The branch locus ${\mathcal B}_{0,[n+1]}$ consists of those hyperelliptic Riemann surfaces admitting more conformal automorphisms than the hyperelliptic one. The description of the groups of conformal automorphisms of hyperelliptic Riemann surfaces can be found in \cite{BGG}.  Theorems \ref{branchconexo} and \ref{teoreal} assert the following simple fact.

\begin{coro}
The locus in ${\mathcal H}_{g}$, consisting of those hyperelliptic Riemann surfaces admitting more conformal automorphisms than the hyperelliptic one, is connected if $2g+1$ is divisible by $3$
and it has exactly two connected components otherwise. The real locus in ${\mathcal H}_{g}$ is connected.
\end{coro}

The above result is related to the ones obtained in \cite{CI2} by Costa, Izquierdo and Porto, where they prove that the hyperelliptic branch locus of orientable Klein surfaces of algebraic genus $g \geq 2$ with one boundary component is connected (in the case of non-orientable Klein surfaces they proved that it has $(g+1)/2$ components, if $g$ is odd, and $(g+2)/2$ components otherwise).

\subsubsection{Generalized Fermat curves}
A closed Riemann surface $S$ is called a generalized Fermat curve of type $(k,n)$, where $k,n \geq 2$ are integers, if it admits a group $H \cong C_{k}^{n}$ of conformal automorphisms such that the quotient orbifold $S/H$ has genus zero and exactly $n+1$ cone points, each one necessarily of order $k$; we say that $H$ is a generalized Fermat group of type $(k,n)$. If $(k-1)(n-1)>2$, then in \cite{GHL} it was observed that $S$ is non-hyperelliptic and in \cite{HKLP} it was proved that $S$ has a unique generalized Fermat group of type $(k,n)$. The uniqueness fact, in particular, asserts that ${\mathcal M}_{0,[n+1]}$ can be identified with the moduli space  ${\mathcal F}_{k,n}$ of generalized Fermat curve of type $(k,n)$ and that the branch locus ${\mathcal B}_{0,[n+1]}$ consists of those admitting more conformal automorphisms than the generalized Fermat group of type $(k,n)$.

\begin{coro} For $(k-1)(n-1)>2$, 
the locus in ${\mathcal F}_{k,n}$, consisting of those admitting more conformal automorphisms than the generalized Fermat group of the $(k,n)$, is connected for $n \geq 4$ even and for $n \geq 6$ divisible by $3$, and it has exactly two connected components otherwise. Its real locus is connected for $n \geq 5$ odd, and it is not connected for $n=2r$, $r \geq 5$ odd.
\end{coro}

\begin{rema}
As it was mentioned to us by one of the referees, the results can be applied as well to the more unknown (and difficult to work with) generic $p$-gonal curves, simple generic $p$-gonal curves \cite{CI:Maximal, CI:Maximal2, Lu-Tan, CH:p-gonal}.
\end{rema}

\begin{nota}
Throughout this paper we denote by $C_{n}$ the cyclic group of order $n$, by $C_{n}^{m}$ the direct product of $m$ copies of $C_{n}$, by $D_{m}$ the dihedral group of order $2m$, by ${\mathcal A}_{4}$ and ${\mathcal A}_{5}$ the alternating groups of orders $12$ and $60$, respectively, and by ${\mathfrak S}_{n}$ the symmetric group in $n$ symbols. 
We will denote by the symbol $\varphi(m)$ the Euler function of an integer $m$.
Also, for an integer $r \geq 2$, we set $\omega_{r}=e^{2\pi i/r}$. We use multiplication of permutations from the left.
\end{nota}

%%%%%%%%%%%%%%%%%%%%%%%%%%
%%%%%%%%%%%%%%%%%%%%%%%%%%
\section{Preliminaries}
\subsection{The moduli and Torelli spaces of marked surfaces}
Let $g, n \geq 0$ be integers such that $3g-2+n >0$, 
$S_{0}$ be a closed orientable surface of genus $g \geq 0$ and let $p_{1},\ldots, p_{n+1} \in S_{0}$ be $(n+1)$ fixed points. A marking of $S_{0}$ is a pair $(S,\phi)$, where $S$ is a closed Riemann surface of genus $g$ and $\phi:S_{0} \to S$ is an orientation-preserving homeomorphism. Two markings $(S_{1},\phi_{1})$ and $(S_{2},\phi_{2})$ are equivalent if there is an biholomorphism $\psi:S_{1} \to S_{2}$ such that $\phi_{2}^{-1} \circ \psi \circ \phi_{1}:S_{0} \to S_{0}$ fixes each of the points $p_{j}$ and it is homotopic to the identity relative the set $\{p_{1},\ldots,p_{n+1}\}$. The Teichm\"uller space ${\mathcal T}_{g,n+1}$ is the set of equivalence classes of the above markings, which is known to be a simply-connected complex manifold of dimension $3g-2+n$ \cite{Nag}. Let ${\rm Hom}^{+}(S_{0};\{p_{1},\ldots,p_{n+1}\})$ be the group of orientation-preserving homeomorphisms of $S_{0}$ keeping the set $\{p_{1},\ldots,p_{n+1}\}$ invariant, and let
${\rm Hom}^{+}(S_{0};(p_{1},\ldots,p_{n+1}))$ be its normal subgroup consisting of those orientation-preserving homeomorphisms of $S_{0}$ fixing each of the points $p_{j}$, $j=1,\ldots,n+1$. The subgroup ${\rm Hom}_{0}(S_{0};(p_{1},\ldots,p_{n+1}))$ of ${\rm Hom}^{+}(S_{0};(p_{1},\ldots,p_{n+1}))$ consisting of those being homotopic to the identity relative to the set $\{p_{1},\ldots,p_{n+1}\}$ is a normal subgroup of ${\rm Hom}^{+}(S_{0};\{p_{1},\ldots,p_{n+1}\})$. The quotient groups 
$${\rm Mod}_{g,[n+1]}={\rm Hom}^{+}(S_{0},\{p_{1},\ldots,p_{n+1}\})/{\rm Hom}_{0}(S_{0};(p_{1},\ldots,p_{n+1}))$$ 
$${\rm Mod}_{g,n+1}={\rm Hom}^{+}(S_{0},(p_{1},\ldots,p_{n+1}))/{\rm Hom}_{0}(S_{0};(p_{1},\ldots,p_{n+1}))$$ are, respectively, the modular group and the pure modular group of $(S_{0},\{p_{1},\ldots,p_{n+1}\})$. These groups act properly discontinuously on ${\mathcal T}_{g,n+1}$ as group of holomorphic automorphisms. By results of Royden \cite{Royden} (for $n=-1$ and $g \geq 2$) and Earle-Kra \cite{EK} (for $n \geq 0$ and $2g+n\geq 4$), 
the group ${\rm Mod}_{g,[n+1]}$ is the full group of holomorphic automorphisms of ${\mathcal T}_{g,n+1}$. The quotient spaces ${\mathcal M}_{g,[n+1]}={\mathcal T}_{g,n+1}/{\rm Mod}_{g,[n+1]}$ and ${\mathcal M}_{g,n+1}={\mathcal T}_{g,n+1}/{\rm Mod}_{g,n+1}$ are, respectively, the moduli and the Torelli spaces of $(n+1)$-marked surfaces of genus $g$ \cite{Bers1,Bers2,Rauch}, both being complex orbifolds of dimension $3g-2+n$ (see, for instance, \cite{Nag}). 
The quotient group  ${\rm Mod}_{g,[n+1]}/ {\rm Mod}_{g,n+1} \cong {\mathfrak S}_{n+1}$, acts as a group of (orbifold) automorphisms of ${\mathcal M}_{g,n+1}$ (also called the Torelli modular group) with quotient orbifold ${\mathcal M}_{g,[n+1]}$. Some details on the above can be found, for instance, in \cite{Harvey, H-M, Nag, Schneps}.

%%%%%%%%%%%%%%%%%%%%%%%
\subsection{The moduli and Torelli spaces of marked spheres}\label{Sec:moduliexplicito}
Assume $g=0$ and $n \geq 3$. It is known that,  
${\rm Mod}_{0,n+1}$ acts freely on ${\mathcal T}_{0,n+1}$, so the Torelli space ${\mathcal M}_{0,n+1}$ is a complex manifold of dimension $n-2$. For $n \geq 4$, the branch locus ${\mathcal B}_{0,[n+1]} \subset {\mathcal M}_{0,[n+1]}$ corresponds to the projection of the points of ${\mathcal M}_{0,n+1}$ with non-trivial ${\rm Mod}_{0,[n+1]}/ {\rm Mod}_{0,n+1}$-stabilizer. For $n=3$ the group ${\mathfrak S}_{4} \cong {\rm Mod}_{0,[4]}/{\rm Mod}_{0,4}$ acts in a non-faithful manner on ${\mathcal M}_{0,4}$ (in fact, every point in ${\mathcal M}_{0,4}$ has non-trivial stabilizer) so ${\mathcal B}_{0,[4]}={\mathcal M}_{0,[4]}$.

\begin{rema}
Igusa \cite{Igusa} observed, by using the invariants of the binary sextics,  that ${\mathcal M}_{0,[6]}$ can be seen as the quotient of ${\mathbb C}^{3}$ by the action of the cyclic group of order five $\langle (x,y,z) \mapsto (\omega_{5}x,\omega_{5}^{2}y,\omega_{5}^{3}z)\rangle$, where $\omega_{5}=e^{2 \pi i/5}$. Using invariants of binary quintics, it can also be obtained that ${\mathcal M}_{0,[5]}$ is the quotient of ${\mathbb C}^{2}$ by the cyclic group of order two $\langle (x,y) \mapsto (-x,-y)\rangle$. In \cite{Patterson} it was observed that, for $n \geq 6$, the moduli space ${\mathcal M}_{0,[n+1]}$ cannot be seen as the quotient of ${\mathbb C}^{n-2}$ by the action of a finite linear group.
\end{rema}

Next, we proceed to recall a natural model $\Omega_{n} \subset {\mathbb C}^{n-2}$ for ${\mathcal M}_{0,n+1}$ together with an explicit form of its automorphisms (see, for instance, \cite{Patterson}). 
Let ${\mathcal X}_{n} \subset \widehat{\mathbb C}^{n+1}$ be the configuration space of ordered $(n+1)$-tuples whose coordinates are pairwise different.
Two tuples $(p_{1},\ldots,p_{n+1}), (q_{1},\ldots,q_{n+1}) \in {\mathcal X}_{n}$ are equivalent if there is a M\"obius transformation $M \in {\rm PSL}_{2}({\mathbb C})$ such that $M(p_{j})=q_{j}$, for $j=1,\ldots,n+1$. As for $(p_{1},\ldots,p_{n+1}) \in {\mathcal X}_{n}$, there is a (unique) M\"obius transformation $M$ such that $M(p_{1})=\infty$, $M(p_{2})=0$ and $M(p_{3})=1$, each $(p_{1},\ldots,p_{n+1})$ is equivalent to a unique one of the form $(\infty,0,1,\lambda_{1},\ldots,\lambda_{n-2})$. It follows that the quotient space ${\mathcal X}_{n}/{\rm PSL}_{2}({\mathbb C})$ can be identified with 
$$\Omega_{n}=\{(z_{1},\ldots,z_{n-2}): z_{j} \in {\mathbb C} \setminus \{0,1\}, \; z_{i} \neq z_{j}\} \subset {\mathbb C}^{n-2}.$$

By the uniformization theorem, each point of ${\mathcal T}_{0,n+1}$ is the class of a pair of the form $(\widehat{\mathbb C},\phi)$, where $\phi:\widehat{\mathbb C} \to \widehat{\mathbb C}$ is an orientation-preserving homeomorphism that fixes $\infty, 0, 1$, so we may identify 
${\mathcal M}_{0,n+1}$ with $\Omega_{n}$. The permutation action of ${\mathfrak S}_{n+1}$ on the coordinates of the tuples $(p_{1},\ldots,p_{n+1}) \in {\mathcal X}_{n}$ is transported to the action of a group ${\mathbb G}_{n}$ of holomorphic automorphisms of $\Omega_{n}$, as describe below. 
Let us fix some $\sigma \in {\mathfrak S}_{n+1}$. Each point $(z_{1},\ldots,z_{n-2}) \in \Omega_{n}$ corresponds to the ordered tuple $(p_{1}=\infty,p_{2}=0,p_{3}=1,p_{4}=z_{1},\ldots,p_{n+1}=z_{n-2}) \in {\mathcal X}_{n}$. We now consider the new tuple 
$(p_{\sigma^{-1}(1)},\ldots,p_{\sigma^{-1}(n+1)}) \in {\mathcal X}_{n}$. There is a unique M\"obius transformation $M_{\sigma,\lambda}$ such that $M_{\sigma,\lambda}(p_{\sigma^{-1}(1)})=\infty$, $M_{\sigma,\lambda}(p_{\sigma^{-1}(2)})=0$ and $M_{\sigma,\lambda}(p_{\sigma^{-1}(3)})=1$; this given as 
$$M_{\sigma,\lambda}(x)=\frac{(x-p_{\sigma^{-1}(2)})(p_{\sigma^{-1}(3)}-p_{\sigma^{-1}(1)})}{(x-p_{\sigma^{-1}(1)})(p_{\sigma^{-1}(3)}-p_{\sigma^{-1}(2)})}.$$

As $\left(M_{\sigma,\lambda}(p_{\sigma^{-1}(4)}),\ldots,M_{\sigma,\lambda}(p_{\sigma^{-1}(n+1)})\right) \in \Omega_{n}$,  the map 
$$\Theta_{n}(\sigma):=T_{\sigma}: \Omega_{n} \to \Omega_{n}:
(z_{1},\ldots,z_{n-2}) \mapsto \left(M_{\sigma,\lambda}(p_{\sigma^{-1}(4)}),\ldots,M_{\sigma,\lambda}(p_{\sigma^{-1}(n+1)})\right),$$
 is an holomorphic automorphism of $\Omega_{n}$.  This procedure provides of a surjective homomorphism (we are using multiplication of permutations from the left) 
 $$\Theta_{n}:{\mathfrak S}_{n+1} \to {\mathbb G}_{n}=\langle A,B\rangle : \sigma \mapsto T_{\sigma},$$
 where $A=\Theta_{n}((1,2))$ and $B=\Theta_{n}((1,2,\ldots,n+1))$. It can be checked that 
$$A\left(z_{1},\ldots,z_{n-2}\right)=\left(\frac{1}{z_{1}},\ldots, \frac{1}{z_{n-2}}\right),\; 
B\left(z_{1},\ldots, z_{n-2}\right)= \left(\frac{z_{n-2}}{z_{n-2}-1},\frac{z_{n-2}}{z_{n-2}-z_{1}},\ldots,\frac{z_{n-2}}{z_{n-2}-z_{n-3}}\right),
$$

If $n=3$, then  $K_{3}:=\ker(\Theta_{3})=\{e,(1,2)(3,4),(1,3)(2,4),(1,4)(2,3)\} \cong C_{2}^{2}$; so ${\mathbb G}_{3} \cong {\mathfrak S}_{3}$.
For $n \geq 4$ the kernel of $\Theta_{n}$ is just the trivial group; so ${\mathbb G}_{n} \cong_{\Theta_{n}} {\mathfrak S}_{n+1}$. In fact, this observation asserts that the correct model for ${\mathcal M}_{0,4}$ is the orbifold whose underlying space is $\Omega_{4}$, but each of its points is a conical point of order $4$.

As ${\mathbb G}_{n}$ is the full group of holomorphic automorphisms of $\Omega_{n}$, every antiholomorphic automorphism of it has the form $T \circ J$, where $T \in {\mathbb G}_{n}$ and $J(z_{1},\ldots,z_{n-2})=(\overline{z}_{1},\ldots,\overline{z}_{n-2})$. The symmetries of $\Omega_{n}$ are those anticonformal automorphisms of order two. In particular, $J$ is a symmetry.
As $J$ commutes with every element of ${\mathbb G}_{n}$, the symmetries of $\Omega_{n}$ are those of the form $T \circ J$, where $T^{2}=I$. 
Summarizing all the above is the following.

\begin{lemm}\label{iso}
Let $n \geq 3$, ${\mathbb G}_{n}=\langle A, B\rangle$ and $\Theta_{n}:{\mathfrak S}_{n+1} \to {\mathbb G}_{n}: \sigma \mapsto T_{\sigma}$ be the surjective homomorphism as defined above. Then 
\begin{enumerate}
\item $\ker(\Theta_{3})=\{e,(1,2)(3,4),(1,3)(2,4),(1,4)(2,3)\} \cong C_{2}^{2}$ and ${\mathbb G}_{3} \cong {\mathfrak S}_{3}$ is the full group of holomorphic automorphisms of $\Omega_{3}$;
\item if $n \geq 4$, ${\mathbb G}_{n} \cong_{\Theta_{n}} {\mathfrak S}_{n+1}$ is the full group of holomorphic automorphisms of $\Omega_{n}$. 
\item The antiholomorphic automorphisms of $\Omega_{n}$ are those of the form $T \circ J$, where $T \in {\mathbb G}_{n}$. Those of order two are for which $T^{2}=I$.
\item The quotient orbifold $\Omega_{n}/{\mathbb G}_{n}$ is a model for the moduli space ${\mathcal M}_{0,[n+1]}$.
\end{enumerate}
\end{lemm}

In the rest of the paper we use the model ${\mathcal M}_{0,n+1}=\Omega_{n}$ and we fix a regular branched cover $\pi_{n}:\Omega_{n} \to \Omega_{n}/{\mathbb G}_{n}$ with deck group ${\mathbb G}_{n}$.

%%%%%%%%%%%%%%%%%%%%%%%%%%
%%%%%%%%%%%%%%%%%%%%%%%%%%
\section{On the connectivity of the branch locus for $n \geq 4$}
From now on, we assume $n \geq 4$. The branch locus ${\mathcal B}_{0,[n+1]} \subset \Omega_{n}/{\mathbb G}_{n}$ consists of the images under $\pi_{n}$ of those points with non-trivial ${\mathbb G}_{n}$-stabilizer. For $\lambda=(\lambda_{1},\ldots,\lambda_{n-2}) \in \Omega_{n}$, we set 
$${\mathcal C}_{\lambda}=\{p_{1}=\infty, p_{2}=0, p_{3}=1, p_{4}=\lambda_{1},\ldots,p_{n+1}=\lambda_{n-2}\},$$
we denote by $G_{\lambda}^{+}$ be the group of M\"obius transformations keeping invariant the set ${\mathcal C}_{\lambda}$, and by $G_{\lambda}$ the group generated by $G_{\lambda}^{+}$ and those extended M\"obius transformations (compositions of complex conjugation with a M\"obius transformation) keeping invariant ${\mathcal C}_{\lambda}$ (so either $G_{\lambda}=G_{\lambda}^{+}$ or $[G_{\lambda}:G_{\lambda}^{+}]=2$). As the cardinality of ${\mathcal C}_{\lambda}$ is bigger than three, it follows that $G_{\lambda}$ is finite.
For the generic case,  $G_{\lambda}$ is trivial and in the non-generic case, $G_{\lambda}^{+}$ is isomorphic to either a cyclic group $C_{m}$, a dihedral group $D_{m}$ (of order $2m$), an alternating group ${\mathcal A}_{4}$ or ${\mathcal A}_{5}$ or the symmetric group ${\mathfrak S}_{4}$. If $G_{\lambda} \neq G_{\lambda}^{+}$, then $G_{\lambda}$ is isomorphic to either 
$D_m$, $C_{m} \times C_{2}$, $D_{m} \rtimes C_{2}$, ${\mathcal A}_{4} \times C_{2}$, ${\mathcal A}_{5} \times C_{2}$, ${\mathfrak S}_{4}$ or ${\mathfrak S}_{4} \times C_{2}$.

As already seen in the previous section, for each $\sigma \in {\mathfrak S}_{n+1}$, such that $\Theta_{n}(\sigma) \in {\mathbb G}_{n}$ fixes $\lambda \in \Omega_{n}$, there is a (unique) M\"obius transformation  $M_{\sigma,\lambda} \in G_{\lambda}^{+}$. In the other direction, 
each $M \in G_{\lambda}^{+}$ induces a permutation $\sigma_{M} \in {\mathfrak S}_{n+1}$ by the following rule:
$$\left(M(p_{\sigma_{M}^{-1}(1)}),\ldots, M(p_{\sigma_{M}^{-1}(n+1)})\right) =\left(p_{1},\ldots,p_{n+1}\right),$$
that is, $M=M_{\sigma_{M},\lambda}$. The above provides of an injective homomorphism 
$$\xi_{\lambda}:G_{\lambda}^{+} \to {\mathfrak S}_{n+1}: M \mapsto \sigma_{M},$$
that, after post-composing it with the isomomorphism $\Theta_{n}:{\mathfrak S}_{n+1} \to {\mathbb G}_{n}$, defines an injective homomorphism
$$\Theta_{n}\circ \xi_{\lambda}: G_{\lambda}^{+} \to {\mathbb G}_{n}: M \mapsto  \Theta_{n}(\xi_{\lambda}(M))=\Theta_{n}(\sigma_{M})$$
whose image is  the ${\mathbb G}_{n}$-stabilizer, ${\rm Stab}_{{\mathbb G}_{n}}(\lambda)$, of the point $\lambda=(\lambda_{1},\ldots,\lambda_{n-2})$.

\begin{rema} \label{formas}
The above (where $n \geq 4$) permits to observe the following facts (see also \cite{Schneps}).
\begin{enumerate}
\item Let $\sigma \in {\mathfrak S}_{n+1}$ be different from the identity permutation and  $T=\Theta_{n}(\sigma) \in {\mathbb G}_{n}$. It follows that $T$ has order $m \geq 2$ and it has fixed points in $\Omega_{n}$ if and only if $\sigma$ is in the conjugacy class of one of the following permutations.
\begin{enumerate}
\item[(1.a)] $(1,2,\ldots,m)(m+1,\ldots,2m) \cdots (rm+1,\ldots,(r+1)m)$, where $n=(r+1)m-1$, some $r \in \{0,1,\ldots\}$.
\item[(1.b)] $(1,2,\ldots,m)(m+1,\ldots,2m) \cdots (rm+1,\ldots,(r+1)m)(n+1)$, where $n=(r+1)m$, some $r \in \{0,1,\ldots\}$.
\item[(1.c)] $(1,2,\ldots,m)(m+1,\ldots,2m) \cdots (rm+1,\ldots,(r+1)m)(n)(n+1)$, where $n=(r+1)m+1$, some $r \in \{0,1,\ldots\}$.
\end{enumerate}

\item
\begin{enumerate}
\item[(2.a)] If $n+1 \equiv \delta \mod(m)$, where $\delta \in \{0,1,2\}$, then we may find $\lambda \in \Omega_{n}$ with ${\rm Stab}_{{\mathbb G}_{n}}(\lambda) \cong C_{m}$.

\item[(2.b)] If $n+1=2mr+m\delta_{1}+\delta_{2}$, where $\delta_{1} \in \{0,1,2\}$ and $\delta_{2} \in \{0,2\}$, then we may find $\lambda \in \Omega_{n}$ with ${\rm Stab}_{{\mathbb G}_{n}}(\lambda) \cong D_{m}$.
 
\item[(2.c)] If $n+1=12r+6\delta_{1}+4\delta_{2}$, where $\delta_{1} \in \{0,1\}$ and $\delta_{2} \in \{0,1,2\}$, then we may find $\lambda \in \Omega_{n}$ with ${\rm Stab}_{{\mathbb G}_{n}}(\lambda) \cong {\mathcal A}_{4}$.
 
\item[(2.d)] If $n+1=24r+12\delta_{1}+8\delta_{2}+6\delta_{3}$, where $\delta_{1},\delta_{2},\delta_{3} \in \{0,1\}$, then we may find $\lambda \in \Omega_{n}$ with ${\rm Stab}_{{\mathbb G}_{n}}(\lambda) \cong {\mathfrak S}_{4}$.

\item[(2.e)] If $n+1=60r+30\delta_{1}+20\delta_{2}+12\delta_{3}$, where $\delta_{1},\delta_{2},\delta_{3} \in \{0,1\}$, then we may find $\lambda \in \Omega_{n}$ with ${\rm Stab}_{{\mathbb G}_{n}}(\lambda) \cong {\mathcal A}_{5}$.

\end{enumerate}
\end{enumerate}
\end{rema}

\begin{example}\label{n=4}
For $n=4$, consider the order five automorphism $B=\Theta_{4}(\sigma)$, where  $\sigma=(1,2,3,4,5)$. Then, 
$$B(z_{1},z_{2})=\left(\frac{z_{2}}{z_{2}-1},\frac{z_{2}}{z_{2}-z_{1}}\right) \in {\mathbb G}_{4},$$
$${\rm Fix}(B)=\left\{ \lambda:=\left( \frac{1+\sqrt{5}}{2}, \frac{3+\sqrt{5}}{2}\right), \; \mu:=\left(\frac{1-\sqrt{5}}{2}, \frac{3-\sqrt{5}}{2} \right) \right\} \subset \Omega_{4}.$$

The order four element 
$$S(z_{1},z_{2})=\left(\frac{1}{1-z_{2}},\frac{z_{1}-1}{z_{2}-1}\right) \in {\mathbb G}_{4}$$
satisfies that $S \circ B \circ S^{-1}=B^{3}$ and it permutes $\lambda$ with $\mu$. Each of these two points is stabilized by the dihedral group $\langle B, S^{2}\rangle \cong D_{5}$.

\end{example}

As observed in the above example, for an element $T \in {\mathbb G}_{n}$ different from the identity and with fixed points in $\Omega_{n}$, it might happen that its locus of fixed points is non-connected. But the two components (two points) are ${\mathbb G}_{4}$-equivalent. In \cite{Schneps} Schneps proved that the connected components of the locus of fixed points of $T$ (each one a complex submanifold) forms an orbit under the action of the normalizing subgroup of $\langle T \rangle$ in ${\mathbb G}_{n}$ (for completeness, we provide a sketch of the proof since in \cite{Schneps} it is explicitly given only one of the cases).

\begin{theo}\label{teo3}
For $n \geq 4$, let  $\Theta_{n}(\sigma)=T \in {\mathbb G}_{n}$, of order $m \geq 2$ and ${\rm Fix}(T) \neq \emptyset$. Let $r+1$, where 
$r \geq 0$, be  the number of cycles of length $m$ in the decomposition of $\sigma$ (as in Remark \ref{formas}). Then the following hold.
\begin{enumerate}

\item Each connected component of ${\rm Fix}(T)$ is a complex submanifold of $\Omega_{n}$ of dimension $r$.

\item If $m=2$, then ${\rm Fix}(T)$ is connected.

\item If $m \geq 3$, then ${\rm Fix}(T)$ has exactly: (i) $\varphi(m)/2$ connected components if $m$ divides $n+1$, and (ii) $\varphi(m)$ connected components otherwise. Moreover, if 
$F_{1}$ and $F_{2}$ are any two of the connected components, then there is an element $S \in {\mathbb G}_{n}$, normalizing $\langle T \rangle$, such that $S(F_{1})=F_{2}$. 

\end{enumerate}
\end{theo}
\begin{proof}
Let $\sigma \in {\mathfrak S}_{n+1}$ be such that $T=\Theta_{n}(\sigma)$.  Up to conjugation, we may assume that $\sigma$ has one of the forms (see (1) of Remark \ref{formas})

\begin{enumerate}
\item $\sigma=(1,2,\ldots,m) \cdots (rm+1,\ldots,(r+1)m)$, if $n=(r+1)m-1$.
\item $\sigma=(1,2,\ldots,m) \cdots (rm+1,\ldots,(r+1)m)(n+1)$, if $n=(r+1)m$.
\item $\sigma=(1,2,\ldots,m) \cdots (rm+1,\ldots,(r+1)m)(n)(n+1)$, if $n=(r+1)m+1$.
\end{enumerate}

Note that, for $m \geq 3$, only one of these possibilities may happen. For $m=2$, both cases (1) and (3) happen for $n$ odd and case (2) only happens for $n$ even.

The image of a point $\lambda=(\lambda_{1},\ldots,\lambda_{n-2}) \in \Omega_{n}$ under $T$ is given by 
$$T(\lambda_{1},\ldots,\lambda_{n-2})=\left(M_{\sigma,\lambda}(p_{\sigma^{-1}(4)}),\ldots,M_{\sigma,\lambda}(p_{\sigma^{-1}(n+1)})\right),$$
where $M_{\sigma,\lambda}$ is the (unique) M\"obius transformation with $M_{\sigma,\lambda}(p_{\sigma^{-1}(1)})=\infty$, $M_{\sigma,\lambda}(p_{\sigma^{-1}(2)})=0$, $M_{\sigma,\lambda}(p_{\sigma^{-1}(3)})=1$ and 
$p_{1}=\infty$, $p_{2}=0$, $p_{3}=1$, $p_{4}=\lambda_{1},\ldots, p_{n+1}=\lambda_{n-2}$. Moreover, $\xi_{\lambda}(M_{\sigma,\lambda})=\sigma$. In this way, ${\rm Fix}(T)$ consists of the tuples
$(\lambda_{1},\ldots,\lambda_{n-2}) \in \Omega_{n}$ such that the set ${\mathcal C}_{\lambda}=\{\infty,0,1,\lambda_{1},\ldots,\lambda_{n-2}\}$ is kept invariant under $M_{\sigma,\lambda}$. 

\subsection*{Case $m=2$} Let us consider a point $\lambda \in {\rm Fix}(T)$.
In this case, $M_{\sigma,\lambda}(x)=\lambda_{1}/x$, whose set of fixed points is ${\rm Fix}(M_{\sigma,\lambda})=\left\{\pm \sqrt{\lambda_{1}}\right\}$.

Case (1), that is, $n-1=2r$, where $r \geq 2$. We must have $\lambda_{2j+1}=\lambda_{1}/\lambda_{2j}$, for $j=1,\ldots, (n-3)/2$. So, the locus ${\rm Fix}(T)$ is homeomorphic to $\Omega_{r+2}$ by identifying the tuple $(\lambda_{1},\lambda_{2},\lambda_{3},\ldots, \lambda_{n-2}) \in {\rm Fix}(T)$ with the tuple $(\lambda_{1},\lambda_{3},\lambda_{5},\ldots,\lambda_{n-2}) \in \Omega_{r+2}$.

Case (2), that is, $n-2=2r$, where $r \geq 1$. We must have $\lambda_{2j+1}=\lambda_{1}/\lambda_{2j}$, for $j=1,\ldots, (n-4)/2$ and $\lambda_{n-2} \in \left\{\pm \sqrt{\lambda_{1}}\right\}$. We can move continuously $\lambda_{1}$ around the origin to pass from one of its squre roots to the other.
So, the locus ${\rm Fix}(T)$ is connected and provides a two fold cover of $\Omega_{r+2}$ by projecting the tuple $(\lambda_{1},\lambda_{2},\lambda_{3},\ldots, \lambda_{n-2}) \in {\rm Fix}(T)$ to the tuple $(\lambda_{1},\lambda_{3},\lambda_{5},\ldots,\lambda_{n-3}) \in \Omega_{r+2}$.

Case (3), that is, $n-3=2r$, where $r \geq 1$.  We must have $\lambda_{2j+1}=\lambda_{1}/\lambda_{2j}$, for $j=1,\ldots, (n-5)/2$ and $\lambda_{n-3}, \lambda_{n-2} \in \left\{\pm \sqrt{\lambda_{1}}\right\}$. Similarly as above, we may move continuously $\lambda_{1}$ around the origin to pass from one of its squre roots to the other.
So, the locus ${\rm Fix}(T)$ is again connected and provides a two fold cover of $\Omega_{r+2}$ by projecting the tuple $(\lambda_{1},\lambda_{2},\lambda_{3},\ldots, \lambda_{n-2}) \in {\rm Fix}(T)$ to the tuple $(\lambda_{1},\lambda_{2},\lambda_{4},\ldots,\lambda_{n-4}) \in \Omega_{r+2}$.

\subsection*{Case $m=3$} Let us consider a point $\lambda \in {\rm Fix}(T)$.
In this case, $M_{\sigma,\lambda}(x)=1/(1-x)$, whose set of fixed points is ${\rm Fix}(M_{\sigma,\lambda})=\left\{(1\pm i\sqrt{3})/2\right\}$. 

Case (1), that is, $n-2=3r$, where $r \geq 1$. We must have $\lambda_{3j-2}=1/(1-\lambda_{3j})$ and $\lambda_{3j-1}=(\lambda_{3j}-1)/\lambda_{3j}$, for $j=1,\ldots, (n-2)/3$. So, the locus ${\rm Fix}(T)$, in this case, is homeomorphic to $\Omega_{r+2}$ by identifying the tuple $(\lambda_{1},\lambda_{2},\lambda_{3},\ldots, \lambda_{n-2}) \in {\rm Fix}(T)$ with the tuple $(\lambda_{3},\lambda_{6},\lambda_{9},\ldots,\lambda_{n-2}) \in \Omega_{r+2}$.

Case (2), that is, $n-3=3r$, where $r \geq 1$. We must have we must have $\lambda_{3j-2}=1/(1-\lambda_{3j})$ and $\lambda_{3j-1}=(\lambda_{3j}-1)/\lambda_{3j}$, for $j=1,\ldots, (n-3)/3$ and $\lambda_{n-2} \in \left\{(1\pm i\sqrt{3})/2\right\}$. We can identify a tuple $(\lambda_{1},\lambda_{2},\lambda_{3},\ldots, \lambda_{n-2}) \in {\rm Fix}(T)$ with the tuple $(\lambda_{3},\lambda_{6},\lambda_{9},\ldots,\lambda_{n-3}, \lambda_{n-2}) \in \Omega_{r+2} \times \left\{(1\pm i\sqrt{3})/2\right\}$. This provides two connected components, each one homeomorphic with $\Omega_{r+2}$, these being permuted by the generator $A$ in Lemma \ref{iso}.

Case (3), that is, $n-4=3r$, where $r \geq 0$. We must have $\lambda_{3j-2}=1/(1-\lambda_{3j})$ and $\lambda_{3j-1}=(\lambda_{3j}-1)/\lambda_{3j}$, for $j=1,\ldots, (n-4)/3$ and $\lambda_{n-3}, \lambda_{n-2} \in \left\{(1\pm i\sqrt{3})/2\right\}$. We can identify a tuple $(\lambda_{1},\lambda_{2},\lambda_{3},\ldots, \lambda_{n-2}) \in {\rm Fix}(T)$ with the tuple $(\lambda_{3},\lambda_{6},\ldots,\lambda_{n-4},\lambda_{n-3}, \lambda_{n-2}) \in \Omega_{r+2} \times \left\{\left((1+ i\sqrt{3})/2  , (1- i\sqrt{3})/2\right), \left((1- i\sqrt{3})/2  , (1+ i\sqrt{3})/2\right)\right\}$ (where $\Omega_{2}$ is just a singleton). This agains provides two connected components, each one homeomorphic with $\Omega_{r+2}$ which are permuted by the generator $A$ in Lemma \ref{iso}.

\subsection*{Case $m \geq 4$}
Let us consider a point $\lambda \in {\rm Fix}(T)$.
In this case, $M_{\sigma,\lambda}(x)=\lambda_{m-3}/(\lambda_{m-3}-x)$ and its set of fixed points is
$${\rm Fix}(M_{\sigma,\lambda})=\left\{ p_{\sigma,\lambda}^{+}=\frac{\lambda_{m-3}+ \sqrt{\lambda_{m-3}(\lambda_{m-3}-4)}}{2},
p_{\sigma,\lambda}^{-}=\frac{\lambda_{m-3}- \sqrt{\lambda_{m-3}(\lambda_{m-3}-4)}}{2} \right\}.$$

As  $M_{\sigma,\lambda}$, of order $m \geq 3$,  must preserve the set $\{\infty,0,1,\lambda_{1},\ldots,\lambda_{m-3}\}$, there is an $M_{\sigma,\lambda}$-invariant circle $\Sigma$ containing these points. As $\infty, 0, 1 \in \Sigma$, it follows that $\Sigma={\mathbb R} \cup \{\infty\}$ and also that $M_{\sigma,\lambda}$ leaves invariant the upper half-plane ${\mathbb H}$. Let $p_{\sigma,\lambda} \in {\rm Fix}(M_{\sigma,\lambda})$ be the fixed point belonging to the upper half-plane ${\mathbb H}$. 

Let $C_{0}$ (respectively, $C_{1}$) be the arc of circle starting at $p_{\sigma,\lambda}$ and ending at $0$ (respectively, ending at $1$) which is orthogonal to the real line. The angle between these two circles at $p_{\sigma,\lambda}$ is $2 \alpha_{\lambda} \pi/m$, for some $\alpha_{\lambda} \in \{1,2,\ldots,m-1\}$ relatively prime with $m$. This value $\alpha_{\lambda}$ determines uniquely the value of $\lambda_{m-3}=\lambda_{m-3}(\alpha_{\lambda})$.

Let ${\mathcal L}_{m}$ be the set of points in $\{1,2,\ldots,[(m-1)/2]\}$ relatively primes to $m$.
As $M_{\sigma,\lambda}$ sends $\infty$ to $0$ and $0$ to $1$, and it must preserve the orientation on the real line, it follows that $\alpha_{\lambda} \in {\mathcal L}_{m}$. Set ${\rm Fix}_{\alpha_{\lambda}}(T) \subset {\rm Fix}(T)$ the set of of those  $\widetilde{\lambda} \in {\rm Fix}(T)$ with $\alpha_{\widetilde{\lambda}}=\alpha_{\lambda}$ (so $\lambda \in {\rm Fix}_{\alpha_{\lambda}}(T)$).

Case (1), that is, $n+1=m(r+1)$. If $r=0$, then the tuple $(\lambda_{1},\ldots,\lambda_{n-2}) \in {\rm Fix}_{\alpha_{\lambda}}(T)$ is uniquely determined by $\alpha_{\lambda}$. If $r=1$, then  $(\lambda_{1},\ldots,\lambda_{n-2}) \in {\rm Fix}_{\alpha_{\lambda}}(T)$ is uniquely determined by $\alpha_{\lambda}$ and $\lambda_{2m-3} \in \Omega_{3}$. If $r \geq 2$, then 
the tuple $(\lambda_{1},\lambda_{2},\lambda_{3},\ldots, \lambda_{n-2}) \in {\rm Fix}_{\alpha_{\lambda}}(T)$  is uniquely determined 
by the tuple $(\lambda_{2m-3},\lambda_{3m-3},\ldots,\lambda_{(r+1)m-3}=\lambda_{n-2}) \in \Omega_{r+2}$. In this way, ${\rm Fix}_{\alpha_{\lambda}}(T)$ is homeomorphic to $\Omega_{r+2}$ (where $\Omega_{2}$ is just a singleton), and the number of connected components of ${\rm Fix}(T)$ is the cardinality of ${\mathcal L}_{m}$, that is, $\varphi(m)/2$.

Case (2), that is, $n=m(r+1)$. If $r=0$,  the tuple $(\lambda_{1},\ldots,\lambda_{n-2}) \in {\rm Fix}_{\alpha_{\lambda}}(T)$ is uniquely determined by $\alpha_{\lambda}$ and the value of $\lambda_{n-2} \in {\rm Fix}(M_{\sigma,\lambda})$. If $r=1$, then $(\lambda_{1},\ldots,\lambda_{n-2}) \in {\rm Fix}_{\alpha_{\lambda}}(T)$ is uniquely determined by $\alpha_{\lambda}$, $\lambda_{2m-3} \in \Omega_{3}$, and $\lambda_{n-2} \in {\rm Fix}(M_{\sigma,\lambda})$.
If $r \geq 2$, then the tuple $(\lambda_{1},\lambda_{2},\lambda_{3},\ldots, \lambda_{n-2}) \in {\rm Fix}_{\alpha_{\lambda}}(T)$ is uniquely determined by 
the tuple $(\lambda_{2m-3},\lambda_{3m-3},\ldots,\lambda_{(r+1)m-3}=\lambda_{n-3}) \in \Omega_{r+2}$ and $\lambda_{n-2} \in  {\rm Fix}(M_{\sigma,\lambda})$. In this way, we obtain that ${\rm Fix}_{\alpha_{\lambda}}(T)$ is homeomorphic to two disjoint copies of $\Omega_{r+2}$. These two components are permuted by the element $\Theta_{n}(\tau)$, where $\tau \in {\mathfrak S}_{n+1}$ is such that $\tau^{-1} \sigma \tau=\sigma^{-1}$ (so, $\Theta_{n}(\tau) \circ T \circ \Theta_{n}(\tau)^{-1}=T^{-1}$). In this way, the number of connected components of ${\rm Fix}(T)$ is two times the cardinality of ${\mathcal L}_{m}$, that is, $\varphi(m)$.

Case (3), that is, $n-1=m(r+1)$.  If $r=0$,  the tuple $(\lambda_{1},\ldots,\lambda_{n-2}) \in {\rm Fix}_{\alpha_{\lambda}}(T)$ is uniquely determined by $\alpha_{\lambda}$, and the value of the pair 
$(\lambda_{n-3},\lambda_{n-2})  \in \left\{(p_{\sigma,\lambda}^{-},p_{\sigma,\lambda}^{+}), (p_{\sigma,\lambda}^{+},p_{\sigma,\lambda}^{-})\right\}$.
If $r=1$, then the tuple $(\lambda_{1},\ldots,\lambda_{n-2}) \in {\rm Fix}_{\alpha_{\lambda}}(T)$ is uniquely determined by $\alpha_{\lambda}$, the value of $\lambda_{2m-3} \in \Omega_{3}$, and
the valuer of the pair 
$(\lambda_{n-3},\lambda_{n-2})  \in \left\{(p_{\sigma,\lambda}^{-},p_{\sigma,\lambda}^{+}), (p_{\sigma,\lambda}^{+},p_{\sigma,\lambda}^{-})\right\}$.
If $r \geq 2$, then the tuple  $(\lambda_{1},\lambda_{2},\lambda_{3},\ldots, \lambda_{n-2}) \in {\rm Fix}_{\alpha_{\lambda}}(T)$ is uniquely determined by the tuple
$(\lambda_{2m-3},\lambda_{3m-3},\ldots,\lambda_{(r+1)m-3}) \in \Omega_{r+2}$ and
$(\lambda_{n-3},\lambda_{n-2})  \in \left\{(p_{\sigma,\lambda}^{-},p_{\sigma,\lambda}^{+}), (p_{\sigma,\lambda}^{+},p_{\sigma,\lambda}^{-})\right\}$.
We obtain that ${\rm Fix}_{\alpha_{\lambda}}(T)$ is homeomorphic to two disjoint copies of $\Omega_{r+2}$. These two components are permuted by an element $\Theta(\tau)$ conjugating $T$ to its inverse (as in the previous case).
Again, the number of connected components of ${\rm Fix}(T)$ is two times the cardinality of ${\mathcal L}_{m}$, that is, $\varphi(m)$.

Let $\lambda, \mu \in {\rm Fix}(T)$ in different connected components. There are integers $\alpha_{\lambda}, \alpha_{\mu} \in {\mathcal L}_{m}$ such that 
$M_{\sigma,\lambda}=R_{\lambda}^{\alpha_{\lambda}}$ and $M_{\sigma,\mu}=R_{\mu}^{\alpha_{\mu}}$, where $R_{\lambda}$ (respectively, $R_{\mu}$) is the M\"obius transformation of order $m$ fixing the points $p_{\sigma,\lambda}$ and $\overline{p_{\sigma,\lambda}}$ (respectively, fixing the points $p_{\sigma,\mu}$ and $\overline{p_{\sigma,\mu}}$) which is rotation at angle $2\pi/m$ at $p_{\sigma,\lambda}$ (respectively, rotation at angle $2\pi/m$ at $p_{\sigma,\mu}$). It follows that there are integers $\beta_{\lambda}, \beta_{\mu} \in \{1,\ldots,m-1\}$, relatively primes to $m$, so that $R_{\lambda}=M_{\sigma,\lambda}^{\beta_{\lambda}}$ and $R_{\mu}=M_{\sigma,\mu}^{\beta_{\mu}}$ (in fact, $\alpha_{\lambda}\beta_{\lambda} \equiv 1 \;{\rm mod}(m)$ and $\alpha_{\mu}\beta_{\mu} \equiv 1 \;{\rm mod}(m)$). The image under $\Theta_{n} \circ \xi_{\lambda}$ of the transformation $R_{\lambda}$ is $T^{\beta_{\lambda}}$
and  the image under $\Theta_{n} \circ \xi_{\mu}$ of $R_{\mu}$ is $T^{\beta_{\mu}}$. As $T^{\beta_{\lambda}}$ and $T^{\beta_{\mu}}$ both generates the cyclic group $\langle T \rangle$, there is an element $S \in {\mathbb G}_{n}$ such that $S \circ T^{\beta_{\lambda}} \circ S^{-1}=T^{\beta_{\mu}}$.
It happens that $S$ sends the set ${\rm Fix}_{\alpha_{\lambda}}(T)$ containing $\lambda$ to the set ${\rm Fix}_{\alpha_{\mu}}(T)$ containing $\mu$.
\end{proof}

\begin{coro}
Let  $T=\Theta_{n}(\sigma) \in {\mathbb G}_{n}$, of order $m \geq 2$, with fixed points in $\Omega_{n}$, where $n \geq 4$, and let $r \geq 0$ be such that in the decomposition of $\sigma$ there are $(r+1)$ cycles of length $m$. Then the projection $\pi_{n}({\rm Fix}(T))={\mathcal B}_{m,r} \subset \Omega_{n}/{\mathbb G}_{n}$  is a connected complex orbifold of dimension $r+2$. 
\end{coro}

\begin{rema}[On Patterson's theorem]\label{Patterson}
Let us assume $n \geq 6$.
Part (1) on the above theorem asserts that the locus of fixed points of a non-trivial element $\Theta_{n}(\sigma) \in {\mathbb G}_{n}$ has dimension $r$, where $\sigma$ is a product of $(r+1)$ disjoint cycles, each one of length $m \geq 2$, with $(r+1)m \in \{n-1,n,n+1\}$. It can be checked that $n-4 \geq r$, so the codimension of the locus of fixed points is at least two and, in particular, that ${\mathcal B}_{0,[n+1]}$ has codimension at least two. It follows from \cite{Prill} that the singular locus of ${\mathcal M}_{0,[n+1]}$ coincides with the branch locus, obtaining Patterson's theorem \cite[Theorem 3]{Patterson}.
\end{rema}

%%%%%%%%%%%%%%%%%%%
%%%%%%%%%%%%%%%%%%%%
\subsection{Proof of Theorem \ref{branchconexo}}\label{Sec:pruebateo1}
\subsection*{(A)} Let us denote by ${\mathcal B}_{2}$ the locus in $\Omega_{n}/{\mathbb G}_{n}$ obtained as the projection of those points being fixed by some involution. We proceed to see that it is a connected set.  For $n \geq 4$ even, there is only one conjugacy class of involutions in ${\mathbb G}_{n}$ with fixed points,  this corresponding to the permutation
$$\sigma=(1,2)(3,4)\cdots(n-1,n)(n+1).$$

So ${\mathcal B}_{2}=\pi_{n}({\rm Fix}(\Theta_{n}(\sigma)))={\mathcal B}_{2,(n-2)/2}$, which is connected.  Let us now assume $n \geq 5$ to be odd. In this case, 
there are two conjugacy classes of involutions in ${\mathbb G}_{n}$ with fixed points, these corresponding to the following two permutations in ${\mathcal S}_{n+1}$
$$\sigma_{1}=(1,2)(3,4)\cdots(n-2,n-1)(n)(n+1)$$
$$\sigma_{2}=\left\{\begin{array}{l}
(1,2s+1)(2,2s+2)\cdots(2s-1,4s-1)(2s,4s)(n,n+1), \; n=4s+1\\
(1,2s+1)(2,2s+2)\cdots(2s-1,4s-1)(2s,4s)(n-2,n-1)(n,n+1),  \; n=4s+3
\end{array}
\right.$$

The involutions $\Theta_{n}(\sigma_{1})$ and $\Theta_{n}(\sigma_{2})$  induce, respectively, the connected sets ${\mathcal B}_{2,(n-3)/2}$ and ${\mathcal B}_{2,(n-1)/2}$ in $\Omega_{n}/{\mathbb G}_{n}$, so ${\mathcal B}_{2}={\mathcal B}_{2,(n-3)/2} \cup {\mathcal B}_{2,(n-1)/2}$. In order to get the connectivity of ${\mathcal B}_{2}$, we proceed to show that ${\mathcal B}_{2,(n-3)/2} \cap {\mathcal B}_{2,(n-1)/2} \neq \emptyset$.

As $\langle \sigma_{1},\sigma_{2}\rangle \cong C_{2}^{2}$, we have that $V_{4}:=\langle \Theta_{n}(\sigma_{1}), \Theta_{n}(\sigma_{2})\rangle \cong C_{2}^{2}$.
 First, let us observe that $[\lambda=(\lambda_{1},\ldots,\lambda_{n-2})] \in {\mathcal B}_{2,(n-1)/2} \cap {\mathcal B}_{2,(n+1)/2}$  if and only if the set 
${\mathcal C}_{\lambda}=\{\infty,0,1,\lambda_{1},\ldots,\lambda_{n-2}\}$ is invariant under  the M\"obius trasnformations  $M_{1}(x)=\lambda_{1}/x$ and $M_{2}(x)=(\lambda_{2s}-\lambda_{2s-2})(x-\lambda_{2s-1})/(\lambda_{2s}-\lambda_{2s-1})(x-\lambda_{2s-2})$ (and none of the points in the set ${\mathcal C}_{\lambda}$ is fixed by $M_{2}$). 
For instance, invariance under $M_{1}$ is guaranteed if $\lambda_{2j}=\lambda_{1}/\lambda_{2j+1}$, for $j=1,\ldots, (n-5)/2$, $\lambda_{n-3}=\sqrt{\lambda_{1}}$ and $\lambda_{n-2}=-\sqrt{\lambda_{1}}$.  In this way, we have freedom in the choices for the parameters $\lambda_{1},\lambda_{3}, \lambda_{5},\ldots, \lambda_{n-6}, \lambda_{n-4}$. Now, assuming the above conditions, $M_{2}$ has order two exactly if $\lambda_{1}-\lambda_{2s-1}\lambda_{2s+1}-\lambda_{2s-1}+\lambda_{2s+1}=0$. Under this extra assumption, we also have that 
$\langle M_{1},M_{2}\rangle \cong C_{2}^{2}$, $M_{2}(0)= \lambda_{2s-1}$ and $M_{2}(\lambda_{1})=\lambda_{2s+1}$. If we set $\lambda_{3}=M_{2}(\lambda_{2s+3}), \ldots, \lambda_{2s-3}=M_{2}(\lambda_{4s-3})$ and, in the case $n=4s+3$, the points $\lambda_{n-4}$ and $\lambda_{n-5}$ are the fixed points of $M_{2}\circ M_{1}$, that $\lambda_{2s-1} \pm \sqrt{\lambda_{2s-1}^{2}-\lambda_{1}}$, then ${\mathcal C}_{\lambda}$ will be invariant under $\langle M_{1},M_{2}\rangle \cong C_{2}^{2}$ as desired.

\subsection*{(B)} Let $T \in {\mathbb G}_{n}$ be of even order $2k$, where $k \geq 1$. If $\lambda \in {\rm Fix}(T)$, then $\lambda$ is fixed under the involution 
$T^{k}$, in particular, $\pi_{n}({\rm Fix}(T))$ intersects ${\mathcal B}_{2}$.

\subsection*{(C)} If $T=\Theta_{n}(\sigma)$ has odd order $m \geq 3$, then we may assume, up to conjugation, that
\begin{enumerate}
\item $\sigma=(1,2,\ldots,m) \cdots (rm+1,\ldots,(r+1)m)$, if $n=(r+1)m-1$.
\item $\sigma=(1,2,\ldots,m) \cdots (rm+1,\ldots,(r+1)m)(n+1)$, if $n=(r+1)m$.
\item $\sigma=(1,2,\ldots,m) \cdots (rm+1,\ldots,(r+1)m)(n)(n+1)$, if $n=(r+1)m+1$.
\end{enumerate}

As before, $\pi_{n}({\rm Fix}(T))={\mathcal B}_{m,r}$. Let $\tau \in {\mathfrak S}_{n+1}$ be the permutation, of order $(r+1)m$, defined as
$$
\begin{array}{c}
\tau(lm+j)=(l+1)m+j, \quad j=1,\ldots,m, \; l=0,\ldots,r-1,\\
\tau(rm+j)=j+1, j=1,\ldots,m-1, \;\tau((r+1)m)=1.
\end{array}
$$

\begin{rema}\label{ojito}
(a) Note that $\Theta_{n}(\tau)$ has a non-empty locus of fixed points, contained inside the locus of fixed points of $T$, and: (i) for $n=(r+1)m-1$, $\tau$ does not fixes any of the symbols, 
(ii) for $n=(r+1)m$, $\tau$ only fixes $n+1$ and (iii) for $n=(r+1)m+1$, $\tau$ only fixes $n$ and $n+1$. 
(b) It can be seen that $\sigma=\tau^{r+1}$, in particular, that ${\mathcal B}_{(r+1)m,0} \cap {\mathcal B}_{m,r} \neq \emptyset$.
(c) If $n \in \{(r+1)m-1, (r+1)m+1\}$, then there is a permutation $\eta \in {\mathfrak S}_{n+1}$ of order two (of the same conjugacy class of either $\sigma_{1}$ or $\sigma_{2}$) such that $\langle \tau,\eta\rangle \cong D_{(r+1)m}$.
\end{rema}

As a consequence of part (c) of Remark \ref{ojito}, if $\sigma$ is as in cases (1) or (3), then ${\mathcal B}_{(r+1)m,0}$ intersects ${\mathcal B}_{2}$. Now, part (b) of the same remark asserts that  ${\mathcal B}_{m,r} \cap {\mathcal B}_{(r+1)m,0} \neq \emptyset$. It follows that the sub-locus of ${\mathcal B}_{0,[n+1]}$, consisiting of the projections under $\pi_{n}$ of the points being fixed by those automorphisms $\Theta_{n}(\sigma)$, where $\sigma$ is either as in (1) or (3), is connected.

In order to obtain connectivity (or not) of ${\mathcal B}_{0,[n+1]}$, we need to study the locus of fixed points  of those automorphisms coming from situation (2) above. So, 
let us assume $n=(r+1)m$ and $\sigma$ as in (2).

\subsubsection*{\bf The case $n \geq 4$ even}
By part (b) of Remark \ref{ojito},  ${\mathcal B}_{m,r} \cap {\mathcal B}_{(r+1)m,0} \neq \emptyset$, and by (B) 
${\mathcal B}_{(r+1)m,0} \cap {\mathcal B}_{2} \neq \emptyset$. All the above then asserts that ${\mathcal B}_{0,[n+1]}$ is connected.

\subsubsection*{\bf The case $n \geq 5$ odd}
In this case, $r \geq 0$ is even and $m \geq 3$ odd. If ${\mathcal B}_{m,r} \cap {\mathcal B}_{2} \neq \emptyset$, then there is a point $\lambda \in {\rm Fix}(T) \cap {\rm Fix}(S)$, where $S=\Theta_{n}(\rho)$, $\rho \in {\mathfrak S}_{n+1}$ is in the same conjugacy class of either $\sigma_{1}$ or $\sigma_{2}$ (so it has no fixed points or exactly two), and $\langle \sigma, \rho\rangle$ being isomorphic to either a cyclic group, a dihedral group, ${\mathcal A}_{4}$, ${\mathcal A}_{5}$ or ${\mathfrak S}_{4}$. The cyclic situation cannot happen as, in this case, $\rho$ should also have only one fixed point, a contradiction. In the dihedral situation, $\rho$ will have to permute two fixed points of $\sigma$, again a contradiction. In the cases ${\mathfrak S}_{4}$ and ${\mathcal A}_{5}$, there should be an involution in $\langle \sigma, \rho\rangle$ permuting two fixed points of $\sigma$, a contradiction. So the only possible situation is 
$\langle \sigma, \rho \rangle \cong {\mathcal A}_{4}$, $m=3$ and  
$n=3(1+2(s+2t))$, for a suitable $s \in \{0,1\}$ and $t \geq 0$, in which case, ${\mathcal B}_{3,2(s+2t)} \cap {\mathcal B}_{2} \neq \emptyset$.
As, by part (b) of Remark \ref{ojito}, ${\mathcal B}_{3,2(s+2t)} \cap {\mathcal B}_{n,0} \neq \emptyset$, we again obtain connectivity of ${\mathcal B}_{0,[n+1]}$ in the case $n$ is divisible by $3$.

In the complementary cases, that is,  for $n \geq 5$ odd, relatively prime to $3$, there is not a permutation in ${\mathfrak S}_{n+1}$ (in the conjugacy class of either $\sigma_{1}$ or $\sigma_{2}$) normalising $\langle \sigma \rangle$, in particular, ${\mathcal B}_{m,r} \cap {\mathcal B}_{2}=\emptyset$, for all possibilities $n=m(r+1)$. As ${\mathcal B}_{n,0}  \cap {\mathcal B}_{m,r} \neq \emptyset$, we obtain that ${\mathcal B}_{0,[n+1]}$ has exactly two connected components (one containing ${\mathcal B}_{2}$ and the other containing ${\mathcal B}_{n,0}$).

%%%%%%%%%%%%%%%%%%%%%%%
%%%%%%%%%%%%%%%%%%%%%%%
\section{On the connectivity of the real locus ${\mathcal M}^{\mathbb R}_{0,[n+1]}$: Proof of Theorem \ref{teoreal}}\label{simetrias}
In this section we proceed to prove Theorem \ref{teoreal}.
As previously noted, a symmetry of $\Omega_{n}$ has the form $T \circ J$, where $T=\Theta_{n}(\sigma) \in {\mathbb G}_{n}$ satisfies that $T^{2}=I$ (that is, $\sigma^{2}$ is the identity permutation). As $J$ commutes with every element of ${\mathbb G}_{n}$, two symmetries $S_{1}=T_{1} \circ J$ and $S_{2}=T_{2} \circ J$ are conjugated by elements of ${\mathbb G}_{n}$ if and only if the elements $T_{1}$ and $T_{2}$ are conjugated. It follows that 
the number of symmetries, up to conjugation by holomorphic automorphisms, is equal to one plus the number of conjugacy classes of elements of order two in the symmetric group ${\mathfrak S}_{n+1}$, that is, $[(n+3)/2]$ (this provides part (1) of Theorem \ref{teoreal}). 
Up to conjugacy, we may assume $$\sigma=(1,2)(3,4)\cdots(2\beta-1,2\beta)(2\beta+1)\cdots(n+1), \; \beta\in \{0,1,\ldots,[(n+1)/2]\},$$
where for $\beta=0$ we mean $\sigma$ the identity permutation. In this case,
\begin{equation}\label{formaT}
T(\lambda_{1},\ldots,\lambda_{n-2})=\left\{\begin{array}{ll}

(\lambda_{1},\ldots,\lambda_{n-2}), & \beta=0.\\

\left(\frac{1}{\lambda_{1}}, \frac{1}{\lambda_{2}}, \ldots,
\frac{1}{\lambda_{n-2}}\right), & \beta=1.\\

\left(\lambda_{1},\frac{\lambda_{1}}{\lambda_{2}}, \frac{\lambda_{1}}{\lambda_{3}}, \ldots,
\frac{\lambda_{1}}{\lambda_{n-2}}\right), & \beta=2.\\

\left(\lambda_{1},\frac{\lambda_{1}}{\lambda_{3}}, \frac{\lambda_{1}}{\lambda_{2}}, \ldots,
\frac{\lambda_{1}}{\lambda_{2s+1}}, \frac{\lambda_{1}}{\lambda_{2s}}, \frac{\lambda_{1}}{\lambda_{2s+2}}, \frac{\lambda_{1}}{\lambda_{2s+3}}, \frac{\lambda_{1}}{\lambda_{n-2}}\right), & s=\beta-2, \; \beta\geq 3.\\
\end{array}
\right.
\end{equation}

and 
\begin{equation}\label{formaS}
S(\lambda_{1},\ldots,\lambda_{n-2})=\left\{\begin{array}{ll}

(\overline{\lambda}_{1},\ldots,\overline{\lambda}_{n-2}), & \beta=0.\\

\left(\frac{1}{\overline{\lambda}_{1}}, \frac{1}{\overline{\lambda}_{2}}, \ldots,
\frac{1}{\overline{\lambda}_{n-2}}\right), & \beta=1.\\

\left(\overline{\lambda}_{1},\frac{\overline{\lambda}_{1}}{\overline{\lambda}_{2}}, \frac{\overline{\lambda}_{1}}{\overline{\lambda}_{3}}, \ldots,
\frac{\overline{\lambda}_{1}}{\overline{\lambda}_{n-2}}\right), & \beta=2.\\

\left(\overline{\lambda}_{1},\frac{\overline{\lambda}_{1}}{\overline{\lambda}_{3}}, \frac{\overline{\lambda}_{1}}{\overline{\lambda}_{2}}, \ldots,
\frac{\overline{\lambda}_{1}}{\overline{\lambda}_{2s+1}}, \frac{\overline{\lambda}_{1}}{\overline{\lambda}_{2s}}, \frac{\overline{\lambda}_{1}}{\overline{\lambda}_{2s+2}}, \frac{\overline{\lambda}_{1}}{\overline{\lambda}_{2s+3}}, \frac{\overline{\lambda}_{1}}{\overline{\lambda}_{n-2}}\right), & s=\beta-2, \; \beta\geq 3.
\end{array}
\right.
\end{equation}

Let us denote by ${\rm Fix}(S) \subset \Omega_{n}$ the locus of fixed points of a symmetry $S$.
 The real locus ${\mathcal M}_{0,[n+1]}^{\mathbb R} \subset \Omega_{n}/{\mathbb G}_{n}$ is the union of all the $\pi_{n}$-images of these fixed sets. Set ${\mathcal F}_{0}=\pi_{n}({\rm Fix}(J))$.

\begin{propo}\label{novacio}
If $S$ is a symmetry of $\Omega_{n}$, then ${\rm Fix}(S) \neq \emptyset$ and every connected component of ${\rm Fix}(S)$ is a real submanifold, of dimension $n-2$. 
\end{propo}
\begin{proof}
Up to conjugation by a suitable element of ${\mathbb G}_{n}$, we may assume $S=T \circ J$, where $T$ and $S$ have the forms as in \eqref{formaT} and \eqref{formaS}, respectively. In this way, $\lambda=(\lambda_{1},\ldots,\lambda_{n-2}) \in \Omega_{n}$ is a fixed point of $S$ if and only if $T(\lambda)=\overline{\lambda}$. Now, as ${\rm Fix}(J)=\Omega_{n} \cap {\mathbb R}^{n-2} \neq \emptyset$, we only need to take care of the case when $T$ is different from the identity (so of order two). 
Let $\lambda=(\lambda_{1},\ldots,\lambda_{n-2}) \in \Omega_{n}$. If $\beta=1$, then $\lambda \in {\rm Fix}(S)$ if and only if $|\lambda_{j}|=1$, $j=1,\ldots n-2$.  
If $\beta=2$,  then $\lambda \in {\rm Fix}(S)$ if and only if $\lambda_{1} \in (0,+\infty)\setminus\{1\}, \; |\lambda_{j}|=\sqrt{\lambda_{1}}, j=2,\ldots, n-2.$ 
If $\beta \geq 3$, then $\lambda \in {\rm Fix}(S)$ if and only if $\lambda_{1} \in (0,+\infty)\setminus\{1\}$, 
$\overline{\lambda}_{3}=\frac{\lambda_{1}}{\lambda_{2}}, \overline{\lambda}_{5}=\frac{\lambda_{1}}{\lambda_{4}}, \ldots,\overline{\lambda}_{2s+1}=\frac{\lambda_{1}}{\lambda_{2s}}$, and $|\lambda_{j}|=\sqrt{\lambda_{1}}$, $j=2s+2,\ldots, n-2$.
As in any of the above situations, the equations on the coordinates have solution, so we are done (see also Remark \ref{puntosfijos}). 
\end{proof}

\begin{rema}[Fixed points description]\label{puntosfijos}
The above proof also permits to obtain a description of the locus of fixed points of the symmetries of $\Omega_{n}$.
For each $\lambda=(\lambda_{1},\ldots,\lambda_{n-2}) \in \Omega_{n}$ we set ${\mathcal C}_{\lambda}=\{p_{1}=\infty, p_{2}=0,p_{3}=1, p_{4}=\lambda_{1},\ldots, p_{n+1}=\lambda_{n-2}\}$. Let us consider a symmetry $S=\Theta_{n}(\sigma) \circ J$, where $\sigma \in {\mathfrak S}_{n+1}$ is either the identity or a permutation of order two. Then

(1) If $\sigma$ is the identity, that is, $S=J$, then $\lambda \in {\rm Fix}(S)$ if and only if ${\mathcal C}_{\lambda} \subset {\mathbb R}\cup \{\infty\}$, that is, ${\mathcal C}_{\lambda}$ is point-wise fixed by the usual complex conjugation map $x \mapsto \overline{x}$. In this case, the connected components of fixed points corresponds to all possible orderings that the collection $\{\lambda_{1},\ldots, \lambda_{n-2}\}$ has in ${\mathbb R}-\{0,1\}$. To be more precise, let ${\mathcal L}$ be the collection of triples $(I_{1},I_{2},I_{3})$, where $I_{1}=(i_{1},\ldots,i_{a})$, $I_{2}=(i_{a+1},\ldots,i_{a+b})$, $I_{3}=(i_{a+b+1},\ldots,i_{n-2})$ and $\{i_{1},\ldots, i_{n-2}\}=\{1,\ldots,n-2\}$ (we permit some of them to be empty tuples). For each tuple $(I_{1},I_{2},I_{3}) \in {\mathcal L}$ we let $L(I_{1},I_{2},I_{3})$ be the set of points $(\lambda_{1},\ldots,\lambda_{n-2}) \in {\rm Fix}(J)=\Omega_{n} \cap {\mathbb R}^{n-2}$ such that $\lambda_{i_{1}}<\cdots<\lambda_{i_{a}}<0<\lambda_{i_{a+1}}<\cdots<\lambda_{i_{a+b}}<1<\lambda_{i_{a+b+1}}<\cdots<\lambda_{i_{n-2}}$. We may observe that ${\rm Fix}(J)$ is the disjoint union of all the sets $L(I_{1},I_{2},I_{3})$, where $(I_{1},I_{2},I_{3}) \in {\mathcal L}$. Observe that, for a given tuple $(I_{1},I_{2},I_{3}) \in {\mathcal L}$ as above, we may find an element $T=\Theta_{n}(\sigma) \in {\mathbb G}_{n}$ (where the permutation $\sigma$ is chosen to keep fix each of the indices $1$, $2$ and $3$)  such that $T(L(I_{1},I_{2},I_{3}))=L((1,\ldots,a),(a+1,\ldots,a+b),(a+b+1,\ldots,n-2)):=L$. Now, given a point $(\lambda_{1},\ldots, \lambda_{n-2}) \in L$, we have the ordered collection
$$\lambda_{1}<\cdots<\lambda_{a}<0<\lambda_{a+1}<\cdots<\lambda_{a+b}<1<\lambda_{a+b+1}<\cdots<\lambda_{n-2}.$$
We may find a M\"obius transformation in ${\rm PSL}_{2}({\mathbb R})$ sending $\lambda_{n-4}$ to $0$, $\lambda_{n-3}$ to $1$ and $\lambda_{n-2}$ to $\infty$. Such a M\"obius transformation induces an element $T \in {\mathbb G}_{n}$ that sends $L$ to $L((1,2,\ldots,n-2),\emptyset,\emptyset)$. This permits to observe that all the connected components of ${\rm Fix}(J)$ are ${\mathbb G}_{n}$-equivalent.

(2) If $\sigma$ has order two, it is a product of $\beta \geq 1$ disjoint transpositions, $2\beta<n+1$, and fixes each of the points  $\{j_{1},\ldots,j_{n+1-2\beta}\} \subset \{1,\ldots,n+1\}$,  then $\lambda \in {\rm Fix}(S)$ if and only if there is a reflection (that is, conjugated to $z \mapsto \overline{z}$) keeping invariant the set ${\mathcal C}_{\lambda}$ and fixing exactly the $n+1-2\beta$ points $p_{j_{1}},\ldots p_{j_{n+1-2\beta}}$. In this case, the connected components of fixed points corresponds to all possible ordering that the collection $\{p_{j_{1}},\ldots p_{j_{n+1-2\beta}}\}$ has in the circle of fixed points of the reflection.

(3) If $n \geq 5$ is odd, $2\beta=n+1$, and $\sigma$ is a product of $\beta$ disjoint transpositions,  
then $\lambda \in {\rm Fix}(S)$ if and only if there is either an imaginary reflection (that is, conjugated to $z \mapsto -1/\overline{z}$) or a reflection keeping invariant the set ${\mathcal C}_{\lambda}$ (and the reflection fixing none of them). By considering the model of $S$ as in \eqref{formaS}, we observe that ${\rm Fix}(S)$ has exactly three connected components: 

$$A_{1}:=\left\{(\lambda_{1},\ldots, \lambda_{n-2}) \in \Omega_{n}: \lambda_{1} \in (-\infty,0), \; \lambda_{2k+1}=\frac{\lambda_{1}}{\overline{\lambda}_{2k}} , k=1,\ldots, (n-3)/2\right\}$$
$$A_{2}:=\left\{(\lambda_{1},\ldots, \lambda_{n-2}) \in \Omega_{n}: \lambda_{1} \in (0,1), \; \lambda_{2k+1}=\frac{\lambda_{1}}{\overline{\lambda}_{2k}} , k=1,\ldots, (n-3)/2\right\}$$
$$A_{3}:=\left\{(\lambda_{1},\ldots, \lambda_{n-2}) \in \Omega_{n}: \lambda_{1} \in (1,\infty), \; \lambda_{2k+1}=\frac{\lambda_{1}}{\overline{\lambda}_{2k}} , k=1,\ldots, (n-3)/2\right\}$$

The component $A_{1}$ corresponds to the imaginary reflection case and the others two, $A_{2}$ and $A_{3}$, to the reflection one. The automorphism
$L(\lambda_{1},\ldots,\lambda_{n-2})=(\lambda_{1}^{-1},\lambda_{2}^{-1},\ldots, \lambda_{n-2}^{-1}) \in {\mathbb G}_{n}$ normalizes $S$ and permutes $A_{2}$ with $A_{3}$.
We may observe that inside each $A_{j}$ there are points with all  of its coordinates being real, in particular, $A_{j} \cap {\rm Fix}(J) \neq \emptyset$. It follows, from Proposition \ref{novacio}, that $\pi_{n}({\rm Fix}(S))$ consists of exactly two real analytic submanifolds $\pi_{n}(A_{1})$ and $\pi_{n}(A_{2})=\pi_{n}(A_{3})$, each one of dimension $n-2$, each one intersecting ${\mathcal F}_{0}$.

\end{rema}

\begin{propo}\label{reales}
Let $S=\Theta_{n}(\sigma) \circ J$ be a symmetry of $\Omega_{n}$, where $n \geq 4$, and let $\beta \in \{0,1,\ldots,[(n+1)/2]\}$ be such that $\sigma$ is the product of $\beta$ transpositions.

\begin{enumerate}
\item If $2\beta \neq n+1$ and $F_{1}$ and $F_{2}$ are any two connected components of the locus of fixed points of $S$, then there is an element $L \in {\mathbb G}_{n}$, normalizing $S$, such that $L(F_{1})=F_{2}$. In particular, the locus ${\mathcal F}_{\beta}:=\pi_{n}({\rm Fix}(S))$ is a connected real orbifold of dimension $n-2$.

\item If $2\beta=n+1$, then ${\rm Fix}(S)$ consists of three connected components, $A_{1}$, $A_{2}$ and $A_{3}$ (as described in Remark \ref{puntosfijos}). There is an element $L \in {\mathbb G}_{n}$, of order two and normalizing $S$, permuting  the two components $A_{2}$ and $A_{3}$. There is no element of ${\mathbb G}_{n}$ that normalizes $S$ and sending $A_{1}$ to  any of the other two. Each $A_{j}$ intersects ${\rm Fix}(J)$. In particular, $\pi_{n}({\rm Fix}(S))$ consists of two connected real orbifolds of dimension $n-2$, say $\pi_{n}(A_{1})$ and ${\mathcal F}_{(n+1)/2}:=\pi_{n}(A_{2})=\pi_{n}(A_{3})$, each of them intersection ${\mathcal F}_{0}$.

\item If $n \geq 5$ is odd, then $\pi_{n}(A_{1}) \cap {\mathcal F}_{\beta} \neq \emptyset$.
\end{enumerate}
\end{propo}
\begin{proof}
Up to conjugation, we may assume $S$ to be as in \eqref{formaS}.
Part (1), for the case $\beta=0$ (respectively, part (2)) was already observed in part (1) (respectively, part (3)) of Remark \ref{puntosfijos}. 

Let us prove part (1) for $\beta >0$. In the case $\beta=1$, we may see that the different connected components of ${\rm Fix}(S)$ correspond to the many different ways to display the values $\lambda_{1},\ldots, \lambda_{n-2}$ in the unit circle. But, by considering permutations of the form $\tau=(1)(2)(3)\widehat{\tau} \in {\mathfrak S}_{n+1}$, we may see that $\Theta_{n}(\tau)$ normalises the symmetry $S$ and permutes these connected components.  The situation is similar for cases $\beta=2$ and $\beta \geq 3$. In the first case we need to use the permutations of the form $\tau=(1)(2)(3)(4)\widehat{\tau}, \tau=(1)(2)(3,4)\widehat{\tau} \in {\mathfrak S}_{n+1}$ and in the second one case we need to use permutations of the form $\tau=(1)(2)(3)(4) \tau_{1} \tau_{2}, \tau=(1)(2)(3,4) \tau_{1} \tau_{2} \in {\mathfrak S}_{n+1}$, where $\tau_{1}$ is the identity permutation on the set $\{5,\ldots,2\beta\}$ and $\tau_{2}$ a permutation on the set $\{2\beta+1,\ldots,n+1\}$. 

Part (3) can be checked just by considering the Klein group $G=\langle U(z)=-1/\overline{z}, V(z)=1/\overline{z}\rangle \cong C_{2}^{2}$. Then we only need to observe that it is possible to find a $G$-invariant collection of $n+1$ points with the property that $n+1-2\beta$ are fixed under the reflection $V$ and the other $2\beta$ are permuted under it. So the result follows from the fixed point description in Remark \ref{puntosfijos}.
\end{proof}

By Proposition \ref{reales} we observe the following. Let $S=\Theta_{n}(\sigma) \circ J$ be a symmetry of $\Omega_{n}$ and $\beta \in \{0,1,\ldots, [(n+1)/2\}$ as above.
\begin{enumerate}
\item If $2\beta \neq n+1$,  then ${\mathcal F}_{\beta}=\pi_{n}({\rm Fix}(S))$ is connected. 

\item If $n \geq 5$ is odd and $2\beta=n+1$, then ${\mathcal F}_{(n+1)/2}:=\pi_{n}(A_{2})=\pi_{n}(A_{3})$ and $\pi_{n}(A_{1})$ are both connected, they intersect and $\pi_{n}({\rm Fix}(S))={\mathcal F}_{(n+1)/2} \cup \pi_{n}(A_{1})$.  

\item If $n \geq 4$ is even, then the real locus ${\mathcal M}_{0,[n+1]}^{\mathbb R}$  is the union the $[(n+3)/2]$ connected real orbifolds ${\mathcal F}_{\beta}$, where $\beta\in \{0,1,\ldots,[(n+1)/2]\}$.

\item If $n \geq 5$ is odd, then the real locus is the union the $(n+1)/2$ connected real orbifolds ${\mathcal F}_{\beta}$, where $\beta\in \{0,1,\ldots,(n+1)/2\}$, together the extra one $\pi_{n}(A_{1})$. The component ${\mathcal F}_{0}$ intersects both ${\mathcal F}_{(n+1)/2}$ and $\pi_{n}(A_{1})$ and, moreover, $\pi_{n}(A_{1})$ intersects all the other ones.
\end{enumerate}

The above asserts that in order to study the connectivity of the real locus, we only need to study the possible intersections between the components
${\mathcal F}_{\beta}$ (for $n$ odd we must also consider the extra component $\pi_{n}(A_{1})$). We call all these sets 
the  ``irreducible" components of ${\mathcal M}_{0,[n+1]}^{\mathbb R}$. The following result provides conditions for two of the irreducible components ${\mathcal F}_{\beta_{1}}$ and ${\mathcal F}_{\beta_{2}}$ to intersect.

\begin{propo}\label{interseccion}
Let $\beta_{1}, \beta_{2} \in \{0,1,\ldots,[(n+1)/2]\}$, $\beta_{1} \neq \beta_{2}$. Then 
${\mathcal F}_{\beta_{1}} \cap {\mathcal F}_{\beta_{2}} \neq \emptyset$  if and only if there are integers $m\geq 1$ and $\gamma \in \{0,1,2\}$ such that  
\begin{equation}\label{formula1}
2m(\beta_{1}+\beta_{2})=(2m-1)(n+1-\gamma).
\end{equation}  
\end{propo}
\begin{proof}
Let us start noting that ${\mathcal F}_{\beta_{1}} \cap {\mathcal F}_{\beta_{2}} \neq \emptyset$ is equivalent (see Remark \ref{puntosfijos}) to have a point $\lambda\in \Omega_{n}$ such that the set ${\mathcal C}_{\lambda}$ is invariant under two reflections, $\tau_{1}$ and $\tau_{2}$, which are non-conjugated by a M\"obius transformation keeping invariant the collection ${\mathcal C}_{\lambda}$ and
\begin{enumerate}
\item[(i)] $\tau_{1}$ fixes pointwise $n+1-2\beta_{1}$ of the points and permutes $2\beta_{1}$ of them;
\item[(ii)] $\tau_{2}$ fixes pointwise $n+1-2\beta_{2}$ of the points and permutes $2\beta_{2}$ of them.
\end{enumerate}

The group $G=\langle \tau_{1}, \tau_{2} \rangle$ is a subgroup of the stabilizer of ${\mathcal C}_{\lambda}$, so it is a finite group; in fact a 
dihedral group of order $2r$, where $r$ is the order of $\tau_{2} \circ \tau_{1}$. 
As $\tau_{1}$ and $\tau_{2}$ are assumed to be non-conjugated, necessarilly $r=2m$, for some $m \geq 1$. 
In this way, there must be non-negative integers $\delta_{1}$ and $\delta_{2}$ and $\gamma \in \{0,1,2\}$, such that on the circle of fixed points of $\tau_{1}$ there are $2\delta_{1}+\gamma$ of the points of ${\mathcal C}_{\lambda}$ and on the circle of fixed points of $\tau_{2}$ we must see $2\delta_{2}+\gamma$ of points of that set, that is (from  first parts of (i) and (ii) above),
$$(*) \quad n+1-2\beta_{1}=2\delta_{1}+\gamma, \; n+1-2\beta_{2}=2\delta_{2}+\gamma,$$
and (from the second part of (i) and (ii)) that
$$(**) \quad 2\beta_{1}=2m\delta_{2}+(2m-2)\delta_{1}, \; 2\beta_{2}=2m\delta_{1}+(2m-2)\delta_{2}.$$

Equalities in $(*)$ impliy that
$$2 \delta_{1}=n+1-2\beta_{1}-\gamma, \; 2 \delta_{2}=n+1-2\beta_{2}-\gamma.$$

Plugging these in the equalities in $(**)$, we obtain the desired result.
\end{proof}

\begin{rema}\label{observa}
For equation \eqref{formula1} to have a solution, necessarily $n+1-\gamma$ must be divisible by $2m$, in particular: (i) for $n$ even, we have $\gamma=1$ and $m$ a divisor of $n/2$, and (ii) for $n$ odd, we have $\gamma \in \{0,2\}$ and $m$ a divisor of $(n+1-\gamma)/2$.  So, for instance,
(1) ${\mathcal F}_{0} \cap {\mathcal F}_{1} = \emptyset$, for  $n \geq 4$, (2) ${\mathcal F}_{1} \cap {\mathcal F}_{2} \neq \emptyset$, if and only if $n\in\{4,5,6,7\}$ and (3)  ${\mathcal F}_{0} \cap {\mathcal F}_{\beta} \neq \emptyset$, if and only if  $\beta \in \{n-1,n,n+1\}$.
\end{rema}

%%%%%%%%%%%%%%%%%
\subsection{${\mathcal M}_{0,[n+1]}^{\mathbb R}$ is connected for $n \geq 5$ odd}

\begin{propo}\label{nimpar}
If $n \geq 5$ is odd, then ${\mathcal M}_{0,[n+1]}^{\mathbb R}$ is connected.
\end{propo}
\begin{proof}
If $\beta \in \{0,\ldots,(n-1)/2\}$, then $(n-1)/2-\beta \in \{0,\ldots,(n-1)/2\}$ and , by using $m=1$ and $\gamma=2$ in \eqref{formula1}, we obtain that 
${\mathcal F}_{\beta} \cap {\mathcal F}_{(n-1)/2-\beta} \neq \emptyset$. Now, by using $m=1$ and $\gamma=0$, we obtain that 
${\mathcal F}_{(n-1)/2-\beta} \cap {\mathcal F}_{\beta+1} \neq \emptyset$. In this way, we may connect using two edges the vertices ${\mathcal F}_{\beta}$ and ${\mathcal F}_{\beta+1}$, for $\beta \in \{0,\ldots,(n-1)/2\}$. Since the component $\pi_{n}(A_{1})$ intersects ${\mathcal F}_{0}$ (in fact, it intersects all the other irreducible components), we obtain the connectivity of ${\mathcal M}_{0,[n+1]}^{\mathbb R}$.
\end{proof}

%%%%%%%%%%%%%%%%%%
\subsection{${\mathcal M}_{0,[n+1]}^{\mathbb R}$ is usually non-connected for $n \geq 4$ even}
In the case $n \geq 4$ even, the connectivity of ${\mathcal M}_{0,[n+1]}^{\mathbb R}$ is described by the  
intersection graph ${\mathcal G}_{n}$ of ${\mathcal M}_{0,[n+1]}^{\mathbb R}$, whose set $V_{n}$ of vertices are the values $\beta \in \{0,1,\ldots, [(n+1)/2]\}$. 
Two different vertices $\beta_{1}, \beta_{2} \in V_{n}$ are joined by an edge if the irreducible components ${\mathcal F}_{\beta_{1}}$ and ${\mathcal F}_{\beta_{2}}$ intersect. The intersection graph ${\mathcal G}_{n}$ describes how the different irreducible components intersect.
Proposition \ref{interseccion} states necessary and sufficient conditions for two different irreducible components to intersect, in particular, it permits to describe the edges of the graph intersection ${\mathcal G}_{n}$.  Some of these graphs are despicted in Figure \ref{n=5}.

\begin{figure}[h]
\begin{center}
\subfigure[${\mathcal G}_{6}$]{\includegraphics[width=3.5cm]{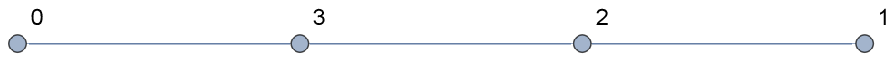}}\quad\quad
\subfigure[${\mathcal G}_{10}$]{\includegraphics[width=3.5cm]{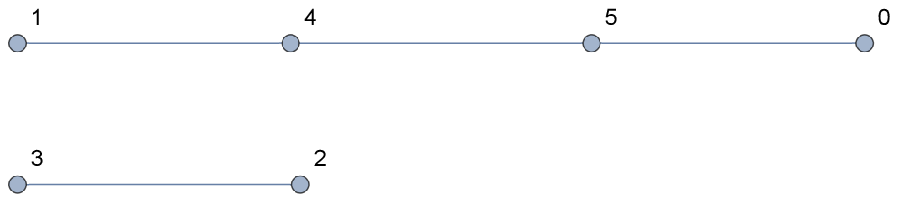}}\quad\quad
\subfigure[${\mathcal G}_{28}$]{\includegraphics[width=3.5cm]{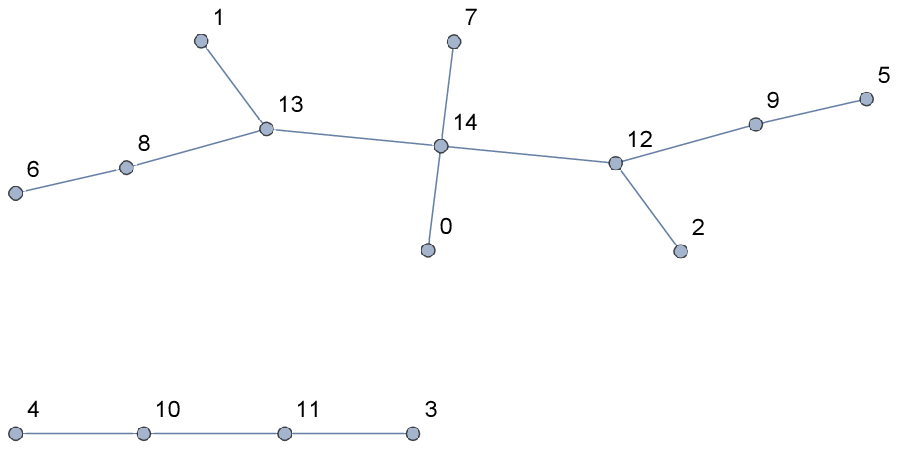}}
\caption{Intersection graphs ${\mathcal G}_{6}, {\mathcal G}_{10}$ and  ${\mathcal G}_{28}$}\label{n=5}
\end{center}
\end{figure}

\begin{propo}[Non-connectedness for $n=2r \geq 4$, $r$ odd]\label{nparporimpar}
If $n=2r$, where $r \geq 5$ is an odd integer, then ${\mathcal M}_{0,[n+1]}^{\mathbb R}$ is not  connected. Moreover, for $r=p$, where $p$ is a prime integer, the real locus ${\mathcal M}_{0,[2p+1]}^{\mathbb R}$ has exactly $(p-1)/2$ connected components.
\end{propo}
\begin{proof} In this case $[(n+1)/2]=r$. By formula \eqref{formula1} and part (i) in Remark \ref{observa}, for $\beta_{1},\beta_{2} \in \{0,1,\ldots,r\}$, $\beta_{1} \neq \beta_{2}$, the condition ${\mathcal F}_{\beta_{1}} \cap {\mathcal F}_{\beta_{2}} \neq \emptyset$ is equivalent to have
$\beta_{1}+\beta_{2}=(2m-1)r/m$, where $m \geq 1$ is a divisor of $r$ (so $m$ must be odd). By taking $m=1$, we obtain that ${\mathcal F}_{(r-1)/2} \cap {\mathcal F}_{(r+1)/2}\neq \emptyset$. We claim that none of these two can intersect other of the components. We check this for $(r-1)/2$ as for the other the argument is similar. Assume ${\mathcal F}_{(r-1)/2}$ intersects ${\mathcal F}_{\beta}$ for some $\beta \neq (r+1)/2$. Then, there must be a divisor $m \geq 1$ of $r$ such that $(r-1)/2+\beta =(2m-1)r/m$.
It follows that $\beta=(m(3r+1)-2r)/2m$, and as $\beta \leq r$, it follows that $m \leq 2r/(r+1)<2$, a contradiction.
If $r=p$, where $p \geq 3$ is a prime, then formula  \eqref{formula1} reads as
$m(\beta_{1}+\beta_{2})=(2m-1)p$, so $m \in \{1,p\}$. In this way, ${\mathcal F}_{\beta_{1}} \cap {\mathcal F}_{\beta_{2}} \neq \emptyset$ if and only if 
$\beta_{1}+\beta_{2}\in\{p,2p-1\}$. 
Using $m=1$, we obtain that ${\mathcal F}_{\beta} \cap {\mathcal F}_{p-\beta} \neq \emptyset$, for every $\beta \in \{0,\ldots, p\}$. By using $m=p$, we obtain that ${\mathcal F}_{\beta} \cap {\mathcal F}_{2p-1-\beta} \neq \emptyset$, for $\beta \in \{p-1,p\}$. It can be seen that $\{0,p,p-1,1\}$ corresponds to one connected component of ${\mathcal M}_{0,[2p+1]}^{\mathbb R}$ and the others correspond to the sets 
$\{2,p-2\}, \{3,p-3\}, \ldots, \{(p-1)/2,(p+1)/2\}$. So, the number of connected components is exactly $(p-1)/2$.
\end{proof}

\begin{rema}
As it was mentioned by one of the referees, it is possible to state a more precise description of the connectivity of ${\mathcal M}_{0,[n+1]}^{\mathbb R}$, for $n=2r$ and $r\geq 5$ odd in a similar way as in Theorem \ref{teo3}. We leave this task to the curious reader.
\end{rema}

\begin{propo}\label{caso4p}
If $n =4p$, where $p \geq 2$ is a prime integer, then ${\mathcal M}_{0,[n+1]}^{\mathbb R}$ is not connected if and only if $p \geq 7$.
\end{propo}
\begin{proof}
If $A, B \subset \{0,\ldots,2p\}$, a map $E:A \to B$ is called a connectivity operator if for $\beta \in A$ we have that ${\mathcal F}_{\beta} \cap {\mathcal F}_{E(\beta)} \neq \emptyset$.
Formula \eqref{formula1} asserts that ${\mathcal F}_{\beta_{1}} \cap {\mathcal F}_{\beta_{2}} \neq \emptyset$, for $\beta_{1},\beta_{2} \in \{0,\ldots,2p\}$, $\beta_{1} \neq \beta_{2}$, if and only if $\beta_{1}+\beta_{2}=(2m-1)2p/m$, where $m \in \{1,2,p,2p\}$.
Each of the values of $m$ induces a connectivity operator as follows
$$\begin{array}{lll}
E_{1}&:&\beta \in \{0,1,\ldots,2p\} \mapsto 2p-\beta \in \{0,1,\ldots,2p\},\\
E_{2}&:&\beta \in \{p,p+1,\ldots,2p\} \mapsto 3p-\beta \in \{p,p+1,\ldots,2p\},\\
E_{p}&:&\beta \in \{2p-2,2p-1,2p\} \mapsto 4p-2-\beta \in \{2p-2,2p-1,2p\},\\
E_{2p}&:&\beta \in \{2p-1,2p\} \mapsto 4p-1-\beta \in \{2p-1,2p\}.
\end{array}
$$

Using the above connectivity operators, it can be checked that, for $k \in \{3,\ldots,p-3\}$ and $p \geq 7$, the vertices in $\{k,2p-k,p+k,p-k\}$ defines a connected component. The connectivity for cases $p \in \{2,3,5\}$ can be checked directly by the connectivity operators.
\end{proof}

\begin{rema}
In the case that $n=4r$, where $r \geq 1$ is odd, but different from a prime, we may use Proposition \ref{interseccion} in order to observe that  ${\mathcal M}_{0,[n+1]}^{\mathbb R}$ is connected for $r=1,9,15,21,27,33$ and it is not connected for $r=25, 35$. In particular, all the above permit to see that there are exactly $32$ values of $n \in \{4,\ldots,100\}$ having not connected real locus, these values being given by 
$10,14,18,22,26,28,30,34,38,42,44,46,50,52,54,58,62,66,68,70,74,76,78,$ 
$82,84,86,88,90,92,94,98,100.$
\end{rema}

%%%%%%%%%%%%%%%%%%%%%%%%%
\begin{rema}[On the field of moduli and fields of definition]\label{observareal}
The group ${\rm Gal}({\mathbb C})$, of field automorphisms of ${\mathbb C}$, acts naturally on $\Omega_{n}$ by the following rule:
if $\nu \in {\rm Gal}({\mathbb C})$ and $\lambda=(\lambda_{1},\ldots,\lambda_{n-2}) \in \Omega_{n}$, then $\nu(\lambda_{1},\ldots,\lambda_{n-2}):=(\nu(\lambda_{1}),\ldots,\nu(\lambda_{n-2}))$. 
The field of moduli ${\mathcal M}_{\lambda}$ of the point $\lambda \in \Omega_{n}$ is the fixed field of the group $\{\nu \in {\rm Gal}({\mathbb C}): \nu(\lambda)=T(\lambda); \; {\rm some} \; T \in {\mathbb G}_{n}\}$.  
A field of definition of $\lambda$ is any subfield ${\mathbb K}$ of ${\mathbb C}$ such that there is an irreducible non-singular projective algebraic curve $X$ of genus zero defined over ${\mathbb K}$ and there is an isomorphism $\psi:\widehat{\mathbb C} \to X$ such that the set $\{\psi(\infty),\psi(0),\psi(1),\psi(\lambda_{1}),\ldots, \psi(\lambda_{n-2})\}$ is invariant under the action of ${\rm Gal}({\mathbb C}/{\mathbb K})$ (the subgroup of all those field automorphisms of ${\mathbb C}$ acting as the identity on ${\mathbb K}$). It can be seen that ${\mathcal M}_{\lambda}$ is contained inside every field of definition of it and that it is the intersection of all its fields of definition \cite{Koizumi}.
In particular, ${\mathcal M}_{\lambda} \leq {\mathbb R}$ if and only $\lambda$ is the fixed point of an antiholomorphic automorphism of $\Omega_{n}$, and 
${\mathbb R}$ is a field of definition of $\lambda$ if and only if  $\{\infty,0,1,\lambda_{1},\ldots,\lambda_{n-2}\}$ is kept invariant under an anticonformal involution of the Riemann sphere, that is, if and only if $\lambda$ is a fixed point of a symmetry of $\Omega_{n}$.
\end{rema}

%%%%%%%%%%%%%%%%
%%%%%%%%%%%%%%%%
\section{Acknowledgments}
The authors would like to thanks to the anonymous referees for their valuable comments and suggestions which permited to improve the presentation of this paper.

%%%%%%%%%%%%%%%%%%%%%%%%
%%%%%%%%%%%%%%%%%%%%%%%%%

\end{document}